\documentclass[reqno,12pt]{amsart}

\pdfoutput=1

\usepackage{color} 
\usepackage{ifpdf}
\ifpdf
    \usepackage[pdftex]{graphicx}
    \usepackage[pdftex]{hyperref}
    \hypersetup{
        unicode=false,          
        pdftoolbar=true,        
        pdfmenubar=true,        
        pdffitwindow=false,     
        pdfstartview={FitH},    
        pdftitle={MCP Article},      
        pdfauthor={Michael Holst},   
        pdfsubject={Mathematics},    
        pdfcreator={Michael Holst},  
        pdfproducer={Michael Holst}, 
        pdfkeywords={PDE, analysis, mathematical physics}, 
        pdfnewwindow=true,      
        colorlinks=true,        
        linkcolor=red,          
        citecolor=blue,         
        filecolor=magenta,      
        urlcolor=cyan           
    }
    
    \def\myfigpdf#1#2{\includegraphics[height=#2]{#1.pdf}}
    
\else
    \usepackage{graphicx}
    \usepackage{pstricks}

\fi

\usepackage{times}
\usepackage{amsfonts}

\usepackage{amsmath}
\usepackage{amsthm}
\usepackage{amssymb}
\usepackage{amsbsy}
\usepackage{amscd}

\usepackage{enumerate}
\usepackage{verbatim}
\usepackage{subfigure}

\newtheorem{theorem}{Theorem}[section]
\newtheorem{corollary}[theorem]{Corollary}
\newtheorem{lemma}[theorem]{Lemma}

\newtheorem{remark}[theorem]{Remark}

\numberwithin{equation}{section}  

  \newcounter{mnote}
  \let\oldmarginpar\marginpar
    \renewcommand\marginpar[1]{\-\oldmarginpar[\raggedleft\footnotesize #1]%
    {\raggedright\footnotesize #1}}

\definecolor{myblue}{rgb}{0.2,0.2,0.7}
\definecolor{mygreen}{rgb}{0,0.6,0}
\definecolor{mycyan}{rgb}{0,0.6,0.6}
\definecolor{myred}{rgb}{0.9,0.2,0.2}
\definecolor{mymagenta}{rgb}{0.9,0.2,0.9}
\definecolor{mywhite}{rgb}{1.0,1.0,1.0}
\definecolor{myblack}{rgb}{0.0,0.0,0.0}

\newcommand{\beq}{\begin{equation}}
\newcommand{\eeq}{\end{equation}}
\newcommand{\beqa}{\begin{eqnarray}}
\newcommand{\eeqa}{\end{eqnarray}}

\def\bold#1{{\mathbf #1}}
\def\text#1{{\mathrm #1}}
\def\calg#1{{\mathcal #1}}
\def\bbbb#1{{\mathbb #1}}

\begin{document}

\title[Schwarz Methods: To Symmetrize or Not to Symmetrize]
      {Schwarz Methods: To Symmetrize or Not to Symmetrize}

\author[M. Holst]{Michael Holst}
\email{holst@ama.caltech.edu}

\author[S. Vandewalle]{Stefan Vandewalle}

\address{Applied Mathematics 217-50, Caltech, Pasadena, CA 91125, USA.}

\thanks{This work was supported in part by the NSF under Cooperative Agreement No.~CCR-9120008.}

\date{August 7, 1995}


\begin{abstract}
A preconditioning theory is presented which establishes sufficient conditions 
for multiplicative and additive Schwarz algorithms to yield self-adjoint 
positive definite preconditioners.
It allows for the analysis and use of non-variational and non-convergent 
linear methods as preconditioners for conjugate gradient methods, and it is 
applied to domain decomposition and multigrid. 
It is illustrated why symmetrizing may be a bad idea 
for linear methods. It is conjectured that enforcing minimal symmetry 
achieves the best results when combined with conjugate gradient acceleration.
Also, it is shown that absence of symmetry in the linear preconditioner 
is advantageous when the linear method is accelerated by using the
Bi-CGstab method.
Numerical examples are presented for two test problems 
which illustrate the theory and conjectures.
\end{abstract}

\maketitle


{\footnotesize
\tableofcontents
}
\vspace*{-0.5cm}

\section{Introduction}
Domain decomposition (DD) and multigrid (MG) methods have been studied 
extensively in recent years, both from a theoretical and numerical point
of view.
DD methods were first proposed in 1869 by H. A. Schwarz as a theoretical 
tool in the study of elliptic problems on non-rectangular 
domains~\cite{Schw69}.
More recently, DD methods have been reexamined for use as practical 
computational tools in the (parallel) solution of general elliptic equations 
on complex domains~\cite{KeXu95}. 
MG methods were discovered much more recently~\cite{Fedo61}. 
They have been extensively developed both theoretically and practically since 
the late seventies~\cite{Bran77,Hack85}, and they have proven to be extremely 
efficient for solving very broad classes of partial differential equations.
Recent insights in the product nature of certain 
MG methods have led to a unified theory of MG and DD methods, 
collectively referred to 
as {\it Schwarz methods}~\cite{BPWX91a,DrWi89,Xu92a}.

In this paper, we consider additive and multiplicative Schwarz methods and 
their acceleration with Krylov methods, for the numerical solution of 
self-adjoint positive definite (SPD) operator equations arising from the 
discretization of elliptic partial differential equations.
The standard theory of conjugate gradient acceleration of linear methods
requires that a certain operator associated with the linear method -- 
the preconditioner -- be symmetric and positive definite.
Often, however, as in the case of Schwarz-based preconditioners, 
the preconditioner is known only implicitly, and
symmetry and positive definiteness are not easily verified.
Here, we try to construct natural sets of sufficient conditions that are
easily verified and do not require the explicit formulation of the
preconditioner.
More precisely, we derive conditions for the constituent components of
MG and DD algorithms (smoother, subdomain solver, transfer operators, etc.),
that guarantee symmetry and positive definiteness of the preconditioning
operator which is (explicitly or implicitly) defined by the resulting 
Schwarz method.

We examine the implications of these conditions for 
various formulations of the standard DD and MG algorithms.
The theory we develop helps to explain the often observed behavior of a 
poor or even divergent MG or DD method which becomes an excellent 
preconditioner when accelerated by a conjugate gradient method.
We also investigate the role of symmetry in linear methods and preconditioners.
Both analysis and numerical evidence suggest that linear methods should not be
symmetrized when used alone, and only minimally symmetrized when 
accelerated by conjugate gradients, in order to achieve the best 
possible convergence results.
In fact, the best results are often obtained when a very nonsymmetric linear 
iteration is used in combination with a nonsymmetric system solver such as 
Bi-CGstab, even though the original problem is SPD.
   
The outline of the paper is as follows. We begin in \S2 by reviewing basic 
linear methods for SPD linear operator equations, and examine Krylov 
acceleration strategies.
In \S3 and \S4, we analyze multiplicative and additive Schwarz preconditioners.
We develop a theory that establishes sufficient conditions for the 
multiplicative and additive algorithms to yield SPD preconditioners.
This theory is used to establish sufficient conditions for multiplicative 
and additive DD and MG methods, and allows for analysis of non-variational and 
even non-convergent linear methods as preconditioners.
A simple lemma, given in \S5, illustrates why symmetrizing may be 
a bad idea for linear methods. 
In \S6, results of numerical experiments obtained with finite-element-based DD 
and MG methods applied to some non-trivial test problems are reported.

\section{Krylov acceleration of linear iterative methods}

In this section, we review some background material on self-adjoint 
linear operators, linear methods, and conjugate gradient acceleration.
More thorough reviews can be found in~\cite{Hack94,Krey78}.

\subsection{Background material and notation}
      \label{subsec:linear_operators}

Let $\calg{H}$ be a real finite-dimensional Hilbert space equipped with the 
inner-product $(\cdot,\cdot)$ inducing the norm 
$\|\cdot\|=(\cdot,\cdot)^{1/2}$.
$\calg{H}$ can be thought of as, for example, the Euclidean space 
$\bbbb{R}^n$, or as an appropriate finite element space.

The {\em adjoint} of a linear operator $A\in\bold{L}(\calg{H},\calg{H})$ with 
respect to $(\cdot,\cdot)$ is the unique operator $A^T$ satisfying
$(Au,v) = (u,A^Tv) \,,~\forall u,v \in \calg{H}$.
An operator $A$ is called {\em self-adjoint} or {\em symmetric} if $A=A^T$;
a self-adjoint operator $A$ is called {\em positive definite} or simply
{\em positive}, if $(Au,u) > 0 \,,~\forall u \in \calg{H}$, ~$u\ne 0$.
If $A$ is self-adjoint positive definite (SPD) with respect to 
$(\cdot,\cdot)$, then the bilinear form 
$(Au,v)$ defines another inner-product on $\calg{H}$, which we denote 
as $(\cdot,\cdot)_A$.
It induces the norm $\|\cdot\|_A = (\cdot,\cdot)_A^{1/2}$.

The adjoint of an operator $M \in \bold{L}(\calg{H},\calg{H})$ with respect 
to $(\cdot,\cdot)_A$, the {\em $A$-adjoint}, 
is the unique operator $M^*$ satisfying 
$~(Mu,v)_A = (u,M^*v)_A \,,~\forall u,v \in \calg{H}$.
From this definition it follows that 
\begin{equation}\label{adjointofM}
M^* = A^{-1} M^T A ~.
\end{equation}
$M$ is called $A$-{\em self-adjoint} if $M=M^*$,
and $A$-{\em positive} if $(Mu,u)_A > 0 \,,~\forall u \in \calg{H}$,
~$u\ne 0$.

If $N \in \bold{L}(\calg{H}_1,\calg{H}_2)$, then the adjoint of $N$,
denoted as $N^T \in \bold{L}(\calg{H}_2,\calg{H}_1)$, 
is defined as the unique operator relating the inner-products in 
$\calg{H}_1$ and $\calg{H}_2$ as follows:
\begin{equation}\label{eqn:adjointofN}
(N u, v)_{\calg{H}_2} = (u, N^T v)_{\calg{H}_1} ~,
 \ \ \ \forall u \in \calg{H}_1 ~, 
 \ \ \ \forall v \in \calg{H}_2 ~.
\end{equation}
Since it is usually clear from the arguments which inner-product is involved, 
we shall often drop the subscripts on inner-products (and norms) throughout the 
paper, except when necessary to avoid confusion.

We denote the spectrum of an operator $M$ as $\sigma(M)$.
The spectral theory for self-adjoint linear operators states that 
the eigenvalues of the self-adjoint operator $M$ are real and lie 
in the closed interval $[\lambda_{\text{min}}(M),\lambda_{\text{max}}(M)]$ 
defined by the Rayleigh quotients:
\begin{equation}\label{RQ}
\lambda_{\text{min}}(M) = \min_{u\ne 0} \frac{(Mu,u)}{(u,u)},
\ \ \ \ \ ~\lambda_{\text{max}}(M) = \max_{u\ne 0} \frac{(Mu,u)}{(u,u)}.
\end{equation}
Similarly, if an operator $M$ is $A$-self-adjoint, then its eigenvalues are 
real and lie in the interval defined by the Rayleigh quotients generated by 
the $A$-inner-product.
A well-known property is that if $M$ is self-adjoint, then the spectral
radius of $M$, denoted as $\rho(M)$, satisfies $\rho(M) = \| M \|$.
This property can also be shown to hold in the $A$-norm for 
$A$-self-adjoint operators 
(or, more generally, for $A$-{\it normal} operators~\cite{AHMS92}).
\begin{lemma}
If $A$ is SPD and $M$ is $A$-self-adjoint, then $~\rho(M) =\| M \|_A$.
   \label{lemma:a-s-a}
\end{lemma}

\subsection{Linear methods}

Given the equation
$Au=f,$
where $A \in \bold{L}(\calg{H},\calg{H})$ is SPD,
consider the {\it preconditioned} equation 
$BAu=Bf$, with $B\in \bold{L}(\calg{H},\calg{H})$.
The operator $B$, the {\em preconditioner}, is usually chosen so 
that a Krylov or Richardson method applied to the preconditioned system
has some desired convergence properties.
A simple linear iterative method employing the operator $B$ takes the form
\begin{equation}\label{linear_method}
u^{n+1} = u^n - BAu^n + Bf = (I - BA)u^n + Bf,
\end{equation}
where the convergence behavior of (\ref{linear_method}) is 
determined by the properties of the so-called {\em error propagation operator},
\begin{equation}
E=I-BA.
\end{equation}
The spectral radius of the error propagator $E$ is called the 
{\it convergence factor} for the linear method, whereas the norm 
is referred to as the {\it contraction number}.
We recall two well-known lemmas; see for example~\cite{Kres89} 
or ~\cite{Orte72}.

\begin{lemma}
For arbitrary $f$ and $u^0$, the condition $\rho(E) < 1$ is necessary and 
sufficient for convergence of the linear method~(\ref{linear_method}).
   \label{lemma:rho}
\end{lemma}

\begin{lemma}
The condition $\|E\|<1$, or the condition $\|E\|_A < 1$,
is sufficient for convergence of the linear method~(\ref{linear_method}).
   \label{lemma:norm}
\end{lemma}

We now state a series of simple lemmas that we shall use repeatedly in 
the following sections.  Their short proofs are added for the reader's
convenience.

\begin{lemma}
   \label{lemma:ONE}
If $A$ is SPD, then $BA$ is $A$-self-adjoint if and only if $B$ is 
self-adjoint.
\end{lemma}
\begin{proof}
Note that: $(ABAu,v)=(BAu,Av)=(Au,B^TAv)$.
The lemma follows since $BA=B^TA$ if and only if $B=B^T$.
\end{proof}

\begin{lemma}
   \label{lemma:TWO}
If $A$ is SPD, then $E$ is $A$-self-adjoint if and only if $B$ is self-adjoint.
\end{lemma}
\begin{proof}
Note that: $(AEu,v)$ $=$ $(Au,v)-(ABAu,v)$ $=$ $(Au,v)-(Au,(BA)^*v)$ $=$ 
$(Au,(I-(BA)^*)v)$. Therefore, $E^* = E$ if and only if $BA$ = $(BA)^*$.
By Lemma~\ref{lemma:ONE}, this holds if and only if $B$ is self-adjoint.
\end{proof}

\begin{lemma}
   \label{lemma:THREE}
If $A$ and $B$ are SPD, then $BA$ is $A$-SPD.
\end{lemma}
\begin{proof}
By Lemma~\ref{lemma:ONE}, $BA$ is $A$-self-adjoint.
Also, we have
$
(ABAu,u)=(BAu,Au)=(B^{1/2}Au,B^{1/2}Au)>0~,~\forall u \ne 0.
$
Hence, $BA$ is $A$-positive, and the result follows.
\end{proof}

\begin{lemma}
If $A$ is SPD and $B$ is self-adjoint, then $\| E \|_A = \rho(E)$.
   \label{lemma:Anorm-eq-rho}
\end{lemma}
\begin{proof}
By Lemma~\ref{lemma:TWO}, $E$ is $A$-self-adjoint. 
By Lemma~\ref{lemma:a-s-a} the result follows.
\end{proof}

\begin{lemma}
If $E^*$ is the $A$-adjoint of $E$, then $\| E \|_A^2 = \| E E^* \|_A$.
   \label{lemma:Anorm_Esquare_EEstar}
\end{lemma}
\begin{proof}
The proof follows that of a familiar result for the 
Euclidean 2-norm~\cite{Hack94}.
\end{proof}

\begin{lemma}
If $A$ and $B$ are SPD, and $E$ is $A$-non-negative, 
then $\| E \|_A < 1$.
   \label{lemma:rho3}
\end{lemma}
\begin{proof}
By Lemma~\ref{lemma:TWO}, $E$ is $A$-self-adjoint. As $E$ is 
$A$-non-negative, it holds $(Eu,u)_A \ge 0$, or $(BAu,u)_A \le (u,u)_A$.
By Lemma~\ref{lemma:THREE}, $BA$ is $A$-SPD, and we have that
$
0 < (BAu,u)_A \le (u,u)_A,~\forall  u \ne 0,
$
which, by (\ref{RQ}), implies that 
$0 < \lambda_i \le 1, ~\forall \lambda_i \in \sigma(BA)$.
Thus, 
$
\rho(E) = 1 - \min_i \lambda_i < 1.
$
Finally, by Lemma~\ref{lemma:Anorm-eq-rho}, we have $\|E\|_A = \rho(E)$.
\end{proof}

We will also have use for the following two simple lemmas. 
\begin{lemma}
   \label{lemma:bounds}
If $A$ is SPD and $B$ is self-adjoint, and $E$ is such that:
$$
-C_1 (u,u)_A \le (Eu,u)_A \le C_2 (u,u)_A, \ \ \ ~\forall u \in \calg{H},
$$
for $C_1\ge 0$ and $C_2\ge 0$, then $\rho(E) = \| E \|_A \le \max\{C_1,C_2\}$.
\end{lemma}
\begin{proof}
By Lemma~\ref{lemma:TWO}, $E$ is $A$-self-adjoint, and by (\ref{RQ})
$\lambda_{min}(E)$ and $\lambda_{max}(E)$
are bounded by $-C_1$ and $C_2$, respectively.
The result then follows by Lemma~\ref{lemma:Anorm-eq-rho}.
\end{proof}
\begin{lemma}
   \label{coro:bounds}
If $A$ and $B$ are SPD, then Lemma~\ref{lemma:bounds} holds for some $C_2<1$.
\end{lemma}
\begin{proof}
By Lemma~\ref{lemma:THREE}, $BA$ is $A$-SPD, which implies that the 
eigenvalues of $BA$ are real and positive.
Hence, we must have that
$\lambda_i(E) = 1 - \lambda_i(BA) < 1,~\forall i$.
Since $C_2$ in Lemma~\ref{lemma:bounds} bounds the largest positive 
eigenvalue of $E$, we have that $C_2 < 1$.
\end{proof}

\subsection{Krylov acceleration of SPD linear methods}

The conjugate gradient method was developed by Hestenes and 
Stiefel~\cite{HeSt52} as a method for solving linear systems 
$Au=f$ with SPD operators $A$.
In order to improve convergence, it is common to 
{\it precondition} the linear system by an SPD 
{\it preconditioning operator} $B \approx A^{-1}$, in which case the 
generalized or preconditioned conjugate gradient method results (\cite{CGO76}).
Our goal in this section is to briefly review some relationships 
between the contraction number of a basic linear preconditioner 
and that of the resulting preconditioned conjugate gradient algorithm.

We start with the well-known conjugate gradient 
contraction bound (\cite{Hack94}):
\begin{equation}
   \label{eqn:cg_error}
\| e^{i+1} \|_A ~\le~ 2 
   \left( 1 - \frac{ 2 }{ 1 + \sqrt{\kappa_A(BA)} } \right)^{i+1}
   \| e^0 \|_A
   ~=~ 2 ~\delta_{\text{cg}}^{i+1} ~\| e^0 \|_A.
\end{equation}
The ratio of extreme eigenvalues of $BA$ appearing in the derivation of the 
bound gives rise to the generalized condition number $\kappa_A(BA)$ appearing 
above.
This ratio is often mistakenly called the (spectral) condition number 
$\kappa(BA)$; in fact, since $BA$ is not self-adjoint, this ratio is not in 
general equal to the usual condition number 
(this point is discussed in great detail in~\cite{AHMS92}).
However, the ratio does yield a condition number in the $A$-norm.
The following lemma is a special case of Corollary~4.2 in~\cite{AHMS92}.
\begin{lemma}
  \label{lemma:cond0}
If $A$ and $B$ are SPD, then
\begin{equation}\label{A_condition_number}
\kappa_A(BA) = \| BA \|_A \| (BA)^{-1} \|_A
    = \frac{\lambda_{\text{max}}(BA)}{\lambda_{\text{min}}(BA)}~.
\end{equation}
\end{lemma}

\begin{remark} \label{remark:alpha}
Often a linear method requires a parameter $\alpha$ in order to 
be convergent, leading to an error propagator of the form
$E=I - \alpha B A$.
Equation (\ref{A_condition_number}) shows that the $A$-condition number does 
not depend on the particular choice of $\alpha$.
Hence, one can use the conjugate gradient method as an accelerator for the 
method without a parameter, avoiding the possibly costly estimation of 
a good $\alpha$.
\end{remark}

The following result gives a bound on the condition number of the operator 
$BA$ in terms of the extreme eigenvalues of the error propagator $E=I-BA$;
such bounds are often used in the analysis of linear preconditioners 
(cf. Proposition~5.1 in~\cite{Xu89}).
We give a short proof of this result for completeness.
\begin{lemma}
  \label{lemma:cond}
If $A$ and $B$ are SPD, and $E$ is such that:
\begin{equation}\label{BA_condition}
-C_1 (u,u)_A \le (Eu,u)_A \le C_2 (u,u)_A, \ \ \ ~\forall u \in \calg{H},
\end{equation}
for $C_1\ge 0$ and $C_2 \ge 0$, then the above must hold with $C_2 < 1$, and 
it follows that: 
$$
\kappa_A(BA) \le \frac{1+C_1}{1-C_2}.
$$
\end{lemma}
\begin{proof}
First, since $A$ and $B$ are SPD, by Lemma~\ref{coro:bounds} we have that 
$C_2 < 1$.
Since $(Eu,u)_A=(u,u)_A-(BAu,u)_A$, it is clear that
$$
(1-C_2) (u,u)_A \le (BAu,u)_A \le (1+C_1) (u,u)_A, 
     \ \ \ ~\forall u \in \calg{H}.
$$
By Lemma~\ref{lemma:THREE}, $BA$ is $A$-SPD. Its eigenvalues are real and 
positive, and lie in the interval defined by the Rayleigh quotients generated 
by the $A$-inner-product.
Hence, that interval is given by $[(1-C_2),(1+C_1)]$, and 
by Lemma~\ref{lemma:cond0} the result follows.
\end{proof}

\begin{remark}
Even if a linear method is not convergent, 
it may still be a good preconditioner.
If it is the case that $C_2 <<1$, and if $C_1>1$ does not become too large, 
then $\kappa_A(BA)$ will be small and the conjugate gradient method will 
converge rapidly, even though the linear method diverges.
\end{remark}

If only a bound on the norm of the error propagator $E=I-BA$ is available,
then the following result can be used to bound the condition number of $BA$.
This result is used for example in~\cite{Xu92a}.
\begin{corollary}
  \label{coro:cond2}
If $A$ and $B$ are SPD, and $\| I - BA \|_A \le \delta < 1$, then
\begin{equation}
 \label{eqn:cond5}
\kappa_A(BA) \le \frac{1 + \delta}{1 - \delta}.
\end{equation}
\end{corollary}
\begin{proof}
This follows immediately from Lemma~\ref{lemma:cond} with 
$\delta = \max\{C_1,C_2\}$.
\end{proof}

The next result connects the contraction number of the 
preconditioner to the contraction number of the 
preconditioned conjugate gradient method. 
It shows that the conjugate gradient method always accelerates a linear method
(if the conditions of the lemma hold).

\begin{lemma}
   \label{theo:accel}
If $A$ and $B$ are SPD, and $\|I-BA\|_A \le \delta < 1$, 
then $\delta_{\text{cg}} < \delta$.
\end{lemma}
\begin{proof}
An abbreviated proof appears in~\cite{Xu92a}, a more detailed proof
in~\cite{Hols94c}.
\end{proof}

\subsection{Krylov acceleration of nonsymmetric linear methods}

The convergence theory of the conjugate gradient iteration requires that the 
preconditioned operator $BA$ be $A$-self-adjoint (see~\cite{AMS90} for more
general conditions), 
which from Lemma~\ref{lemma:ONE} requires that $B$ be self-adjoint.
If a Schwarz method is employed which produces a nonsymmetric operator $B$,
then although $A$ is SPD, the theory of the previous section does not apply,
and a nonsymmetric solver such as conjugate gradients on the normal 
equations~\cite{AMS90}, GMRES~\cite{SaSc86}, CGS~\cite{Sonn89}, or 
Bi-CGstab~\cite{Vors92} must be used for the now non-$A$-SPD preconditioned 
system, $BAu=Bf$.

The conjugate gradient method for SPD problems has several nice properties
(good convergence rate, efficient three-term recursion, 
and minimization of the $A$-norm of the error at each step), 
some of which must be given up in order 
to generalize the method to nonsymmetric problems.
For example, while GMRES attempts to maintain a minimization property and
a good convergence rate, the three-term recursion must be sacrificed.  
Conjugate gradients on the normal equations maintains a minimization property
as well as the efficient three-term recursion, but sacrifices convergence speed
(the effective condition number is the square of the original system).
Methods such as CGS and Bi-CGstab sacrifice the minimization property, but
maintain good convergence speed and the efficient three-term recursion.
For these reasons, methods such as CGS and Bi-CGstab have become the methods 
of choice in many applications that give rise to nonsymmetric problems.  
Bi-CGstab has been shown to be more attractive than CGS in many situations 
due to the more regular convergence behavior~\cite{Vors92}.
In addition, Bi-CGstab does not require the application of the adjoint of
the preconditioning operator, which can be difficult to implement in the case
of some Schwarz methods.

In \S6, we shall use the preconditioned Bi-CGstab 
algorithm to accelerate nonsymmetric Schwarz methods.
In a sequence of numerical experiments, we shall compare the effectiveness
of this approach with unaccelerated symmetric and nonsymmetric Schwarz 
methods, and with symmetric Schwarz methods accelerated with conjugate 
gradients.

\section{Multiplicative Schwarz methods}

We develop a preconditioning theory of product algorithms which 
establishes sufficient conditions for producing SPD preconditioners.
This theory is used to establish sufficient SPD conditions for 
multiplicative DD and MG methods.

\subsection{A product operator}

Consider a product operator of the form:
\begin{equation}\label{product_operator}
E = I - BA = (I - \bar{B}_1 A) (I - B_0 A) (I - B_1 A)~,
\end{equation}
where $\bar{B}_1, B_0$ and $B_1$ are linear operators on $\calg{H}$, and
where $A$ is, as before, an SPD operator on $\calg{H}$.
We are interested in conditions for $\bar{B}_1, B_0$ and $B_1$, 
which guarantee that the implicitly defined operator $B$ is self-adjoint 
and positive definite and, hence, can be accelerated by using the 
conjugate gradient method.
\begin{lemma}
   \label{lemma:main_mult}
Sufficient conditions for symmetry and positivity of operator $B$,
implicitly defined by (\ref{product_operator}), are:
\begin{enumerate}
\item $\bar{B}_1 = B_1^T$~;
\item $B_0 = B_0^T$~;
\item $\| I - B_1 A \|_A < 1$~;
\item $B_0$ non-negative on $\calg{H}$~.
\end{enumerate}
\end{lemma}
\begin{proof}
By Lemma \ref{lemma:TWO}, in order to prove symmetry of $B$, it is 
sufficient to prove that $E$ is $A$-self-adjoint.  
By using (\ref{adjointofM}), we get
\begin{eqnarray*}
E^* & = & A^{-1} E^T A \\
    & = & A^{-1} (I - A B_1^T) (I - A B_0^T) (I - A \bar{B}_1^T) A \\
    & = & (I - B_1^T A) (I - B_0^T A) (I - \bar{B}_1^T A) \\
    & = & (I - \bar{B}_1 A) (I - B_0 A) (I - B_1 A)   =  E,
\end{eqnarray*}
which follows from conditions 1 and 2.

Next, we prove that $(Bu,u)>0$, $\forall u \in \calg{H}$, $u \ne 0$.
Since $A$ is non-singular, this is equivalent to proving that $(BAu,Au) > 0$.
Using condition 1, we have that
\begin{eqnarray*}
(BAu,Au) & = & ((I-E)u,Au) \\
         & = & (u,Au) - ((I-B_1^T A) (I-B_0 A) (I-B_1 A) u, A u) \\
         & = & (u,Au) - ((I-B_0 A) (I-B_1 A) u, A (I-B_1 A) u) \\
         & = & (u,Au) - ((I-B_1 A) u, A (I-B_1 A) u) + ( B_0 w, w),
\end{eqnarray*}
where $w = A(I-B_1 A)u$.
By condition 4, we have that $(B_0 w,w) \ge 0$. 
Condition 3 implies that $((I-B_1 A) u, A(I-B_1A) u) < (u,Au)$ for $u \ne 0$.
Thus, the first two terms in the sum above are together positive, 
while the third one is non-negative, so that $B$ is positive.
\end{proof}

\begin{corollary}\label{corr_symm}
If $B_1 = B_1^T$, then condition~3 in Lemma~\ref{lemma:main_mult} is 
equivalent to $\rho(I-B_1 A) < 1$.
\end{corollary}
\begin{proof}
This follows directly from Lemma~\ref{lemma:a-s-a} and Lemma \ref{lemma:TWO}.
\end{proof}

\subsection{Multiplicative domain decomposition}\label{sect_mdd}

Given the finite-dimensional Hil\-bert space $\calg{H}$, consider $J$ spaces 
$\calg{H}_k,~k=1,\ldots,J$, together with linear operators 
$I_k \in \bold{L}(\calg{H}_k,\calg{H})$, $\mbox{null}(I_k)=\{0\}$,  such that 
$I_k \calg{H}_k \subseteq \calg{H} = \sum_{k=1}^J I_k \calg{H}_k$.
We also assume the existence of another space $\calg{H}_0$, an
associated operator $I_0$ such that $I_0\calg{H}_0 \subseteq \calg{H}$,
and some linear operators 
$I^k \in \bold{L}(\calg{H},\calg{H}_k), k=0,\ldots,J$. 
For notational convenience, we shall denote the inner-products on $\calg{H}_k$
by $(\cdot,\cdot)$ (without explicit reference to the particular space).
Note that the inner-products on different spaces need not be related.

In a domain decomposition context, the spaces $\calg{H}_k$, $k=1,\ldots,J$, 
are typically associated with ~{\em local subdomains}~ of the original domain
on which the partial differential equation is defined. 
The space $\calg{H}_0$ is then a space associated with some global coarse mesh.
The operators $I_k, k=1,\ldots, J$, are usually inclusion operators, while
$I_0$ is an interpolation or prolongation operator 
(as in a two-level MG method). 
The operators $I^k, k=1,\ldots,J$, are usually orthogonal projection operators,
while $I^0$ is a restriction operator 
(again, as in a two-level MG method).

The error propagator of a multiplicative DD method on the space 
$\calg{H}$ employing the subspaces $I_k \calg{H}_k$ has the general 
form~\cite{DrWi89}:
\begin{equation}\label{mult_prop}
E = I - BA = (I - I_J \bar{R}_J I^J A) \cdots
             (I - I_0 R_0 I^0 A) \cdots
             (I - I_J R_J I^J A)~,
\end{equation}
where $\bar{R}_k$ and $R_k$, $k=1,\ldots,J$, are linear operators on 
$\calg{H}_k$, and $R_0$ is a linear operator on $\calg{H}_0$.
Usually the operators $\bar{R}_k$  and $R_k$ are constructed so that 
$\bar{R}_k \approx A_k^{-1}$ and $R_k \approx A_k^{-1}$, 
where $A_k$ is the operator defining the subdomain problem in $\calg{H}_k$. 
Similarly, $R_0$ is constructed so that $R_0\approx A_0^{-1}$. Actually,
quite often $R_0$ is a ``direct solve", i.e., $R_0 = A_0^{-1}$.
The subdomain problem operator $A_k$ is related to the restriction of 
$A$ to $\calg{H}_k$.  
We say that $A_k$ satisfies the {\it Galerkin conditions} or, in a
finite element setting, that it is {\it variationally} defined when
\begin{equation}
A_k = I^k A I_k,
\ \ \ \ \ I^k = I_k^T.
   \label{eqn:variational}
\end{equation}
Recall that the superscript ``$T$'' is to be interpreted as the adjoint in
the sense of~(\ref{eqn:adjointofN}), i.e., with respect to the 
inner-products in 
$\calg{H}$ and $\calg{H}_k$.

In the case of finite element, finite volume, or finite difference 
discretization of an elliptic problem, conditions~(\ref{eqn:variational}) 
can be shown to hold naturally for both the matrices and the abstract weak 
form operators for all subdomains $k = 1, \ldots, J$.
For the coarse space $\calg{H}_0$, often~(\ref{eqn:variational}) 
must be imposed algebraically.

Propagator (\ref{mult_prop}) can be thought of as the product operator 
(\ref{product_operator}), by choosing
$$
I - \bar{B}_1 A = \prod_{k=J}^1 (I - I_k \bar{R}_k I^k A)~,~~
B_0 = I_0 R_0 I^0~,~~
I - B_1 A = \prod_{k=1}^J (I - I_k R_k I^k A)~,
$$
where $\bar{B}_1$ and $B_1$ are known only implicitly.
(Note that we take the convention that the first term in the product appears on 
the left.)
This identification allows for the use of Lemma~\ref{lemma:main_mult} to 
establish sufficient conditions on the subdomain operators $\bar{R}_k$, 
$R_k$ and $R_0$ to guarantee that multiplicative domain decomposition yields 
an SPD operator $B$.
\begin{theorem}
   \label{theo:main_mult_dd}
Sufficient conditions for symmetry and positivity of the multiplicative
domain decomposition operator $B$, implicitly defined by (\ref{mult_prop}), 
are:
\begin{enumerate}
\item $I^k = c_k I_k^T~, ~~c_k > 0~, ~~k=0,\cdots, J$~;
\item $\bar{R}_k = R_k^T~,~~k=1,\cdots, J$~;
\item $R_0 = R_0^T$~;
\item $\left\| \prod_{k=1}^J (I - I_k R_k I^k A) \right\|_A < 1$~;
\item $R_0$ non-negative on $\calg{H}_0$~.
\end{enumerate}
\end{theorem}

\begin{proof}
We show that the sufficient conditions of Lemma \ref{lemma:main_mult} are 
satisfied.
First, we prove that $\bar{B}_1 = B_1^T$, which, by Lemma~\ref{lemma:TWO}, 
is equivalent to proving that $(I-B_1 A)^* = (I - \bar{B}_1 A)$.
By using (\ref{adjointofM}), we have
$$
\left( \prod_{k=1}^J (I - I_k R_k I^k A) \right)^*
 = A ^{-1} \left( \prod_{k=1}^J (I - I_k R_k I^k A) \right)^T A
 = \prod_{k=J}^1 (I - (I^k)^T R^T_k (I_k)^T A)~,
$$
which equals $(I - \bar{B}_1 A)$ under conditions 1 and 2 of the theorem.
The symmetry of $B_0$ follows immediately from conditions 1 and 3; indeed,
$$
B_0^T = (I_0 R_0 I^0)^T = (I^0)^T R_0^T (I_0)^T = 
        (c_0 I_0) R_0 (c_0^{-1} I^0) =
             I_0 R_0 I^0 = B_0~.
$$

By condition~4 of the theorem, condition~3 of Lemma~\ref{lemma:main_mult} holds 
trivially. 
The theorem follows by realizing that condition 4 of 
Lemma~\ref{lemma:main_mult} is also satisfied, since,
$$
(B_0 u,u) = (I_0 R_0 I^0 u,u) 
      = (R_0 I^0 u, I_0^T u) 
      = c_0^{-1} (R_0 I^0 u, I^0 u ) \geq 0~,
  ~~\forall u\in \calg{H}~.
$$
\end{proof}

\begin{remark} \label{remark:dd_mult_galerkin}
Note that one sweep through the subdomains, followed by a coarse problem solve,
followed by another sweep through the subdomains in reversed order, 
gives rise an error propagator of the form (\ref{mult_prop}).
Also, note that no conditions are imposed on the nature of the operators $A_k$
associated with each subdomain.
In particular, the theorem {\it does not} require that the 
variational conditions are satisfied. 
While it is natural for condition~(\ref{eqn:variational}) to hold between the
fine space and the spaces associated with each subdomain, these conditions
are often difficult to enforce for the coarse problem.
Violation of variational conditions can occur, for example, when 
complex coefficient discontinuities do not lie along element boundaries on 
the coarse mesh (we present numerical results for such a problem
in~\S\ref{sec:numerical}).
The theorem also does not require that the overall multiplicative
DD method be convergent.
\end{remark}

\begin{remark} \label{remark:matrices}
The results of the theorem apply for abstract operators on general
finite-dimensional Hilbert spaces with arbitrary inner-products.
They hold in particular for matrix operators on $\bbbb{R}^n$, equipped
with the Euclidean inner-product, or the discrete $L^2$ inner-product.
In the former case, the superscript ``$T$'' corresponds to the standard 
matrix transpose.
In the latter case, the matrix representation of the adjoint is a scalar
multiple of the matrix transpose; the scalar may be different from unity
when the adjoint involves two different spaces, and in the case of
prolongation and restriction.
This possible constant in the case of the discrete $L^2$ inner-product is 
absorbed in the factor $c_k$ in condition 1.
This allows for an easy verification of the conditions of the theorem in
an actual implementation, where the operators are represented as matrices,
and where the inner-products do not explicitly appear in the algorithm.
\end{remark}

\begin{remark}
Condition 1 of the theorem (with $c_k=1$) for $k=1,\ldots,J$ is usually 
satisfied trivially for domain decomposition methods.
For $k=0$, it may have to be imposed explicitly.
Condition 2 of the theorem allows for several alternatives which 
give rise to an SPD preconditioner, namely:
(1) use of exact subdomain solvers (if $A_k$ is a symmetric operator);
(2) use of identical symmetric subdomain solvers in the forward and
    backward sweeps;
(3) use of the adjoint of the subdomain solver on the second sweep.
Condition~3 is satisfied when the coarse problem is symmetric and the
solve is an exact one, which is usually the case.  If not, the coarse
problem solve has to be symmetric.
Condition~4 in Theorem~\ref{theo:main_mult_dd} is clearly a non-trivial one; 
it is essentially the assumption that the multiplicative DD method without a 
coarse space is convergent.
Convergence theories for DD methods can be quite technical and depend on such 
things as the discretization, the subdomain number, shape, and size, and
the regularity of the solution~\cite{BPWX91a,DrWi89,Xu92a}.
However, since variational conditions hold naturally between the fine space
and each subdomain space for nearly any formulation of a DD method, very 
general convergence theorems can be derived, if one is not concerned about
the actual rate of convergence.
Using the Schwarz theory framework in any of~\cite{BPWX91a,DrWi89,Xu92a}, 
it can be shown that Condition~4 in Theorem~\ref{theo:main_mult_dd} 
(convergence of multiplicative DD without a coarse space) holds if the
variational conditions~(\ref{eqn:variational}) hold, and if the subdomain
solvers $R_k$ are SPD.
A proof of this result may be found for example in~\cite{Hols94c}.
Condition 5 is satisfied for example when the coarse problem is SPD and the 
solve is exact.

\end{remark}

Consider now the case when the subspaces together do not 
span the entire space, except when the coarse space is included.
The above theorem can be applied with $R_0=0$, and by viewing the 
coarse space as simply one of the spaces $\calg{H}_k$, $k\ne 0$.
In this case, the error propagation operator $E$ takes the form:
\begin{equation}\label{mult_prop_v2}
~~~~~I - BA = (I - I_J \bar{R}_J I^J A) \cdots
             (I - I_1 \bar{R}_1 I^1 A) (I - I_1 R_1 I^1 A) \cdots
             (I - I_J R_J I^J A)~.
\end{equation}
This leads to the following corollary.
\begin{corollary}
   \label{corr:main_mult_dd_nonoverlap}
Sufficient conditions for symmetry and positivity of the multiplicative
domain decomposition operator $B$, implicitly defined by (\ref{mult_prop_v2}),
are:
\begin{enumerate}
\item $I^k = c_k I_k^T~, ~~c_k > 0~, ~~k=1,\cdots, J$~;
\item $\bar{R}_k = R_k^T~,~~k=1,\cdots, J$~;
\item $\left\| \prod_{k=1}^J (I - I_k R_k I^k A) \right\|_A < 1$~.
\end{enumerate}
\end{corollary}

\vspace*{0.3cm}

\begin{remark}
Condition~3 is equivalent to requiring convergence of the overall 
multiplicative Schwarz method.
This follows from the relationship
$$
\| E \|_A = \| \bar{E}^* \bar{E} \|_A = \| \bar{E} \|_A^2 < 1,
$$
where $\bar{E} = \prod_{k=1}^J (I - I_k R_k I^k A)$.
\end{remark}
\begin{remark}
If, in addition to conditions of the corollary, 
it holds that $R_1 = (I^1 A I_1)^{-1}$, i.e., it corresponds to an 
exact solve with a variationally defined subspace problem operator 
in the sense of~(\ref{eqn:variational}), then 
$$
(I - I_1 \bar{R}_1 I^1 A) (I - I_1 R_1 I^1 A)
= I - I_1 R_1 I^1 A,
$$
since ~$I-I_1 (I^1 A I_1)^{-1} I_1 A$~ is a projector.
Therefore, space $\calg{H}_1$ (for example, the coarse space)
needs to be visited only once in the application of~(\ref{mult_prop_v2}).
\end{remark}

\subsection{Multiplicative multigrid}\label{sect_mmg}

Consider the Hilbert space $\calg{H}$, $J$ spaces $\calg{H}_k$
together with linear operators $I_k \in \bold{L}(\calg{H}_k,\calg{H})$,
$\mbox{null}(I_k)=0$,
such that the spaces $I_k \calg{H}_k$ are nested and satisfy
$
I_1 \calg{H}_1 \subseteq I_2 \calg{H}_2 \subseteq 
     \cdots \subseteq I_{J-1} \calg{H}_{J-1}
            \subseteq \calg{H}_J \equiv \calg{H}.
$
As before we denote the $\calg{H}_k$-inner-products by $(\cdot,\cdot)$, 
since it will be clear from the arguments which inner-product is intended.
Again, the inner-products are not necessarily related in any way.
We assume also the existence of operators 
$I^k \in \bold{L}(\calg{H},\calg{H}_k)$.

In a multigrid context, the spaces $\calg{H}_k$ are 
typically associated with a nested hierarchy of successively refined meshes,
with $\calg{H}_1$ being the coarsest mesh, and $\calg{H}_J$ being
the fine mesh on which the PDE solution is desired.
The linear operators $I_k$ are prolongation operators, constructed 
from given interpolation or prolongation operators that operate between 
subspaces, i.e.,
$I_{k-1}^k \in \bold{L}(\calg{H}_{k-1},\calg{H}_k)$.  
The operator $I_k$ is then constructed (only as a theoretical tool)
as a composite operator
\begin{equation}
I_k = I_{J-1}^J I_{J-2}^{J-1} \cdots I_{k+1}^{k+2} I_k^{k+1},
\ \ \ k = 1, \ldots, J-1.
  \label{eqn:composite}
\end{equation}
The composite restriction operators $I^k$, $k=1,\ldots, J-1$, are constructed 
similarly from some given restriction operators 
$I_k^{k-1} \in \bold{L}(\calg{H}_k,\calg{H}_{k-1})$.

The coarse problem operators $A_k$ are related to the restriction of 
$A$ to $\calg{H}_k$.  
As in the case of DD methods, we say that $A_k$ is {\it variationally} 
defined or satisfies the {\em Galerkin conditions} when 
conditions~(\ref{eqn:variational}) hold.
It is not difficult to see that conditions~(\ref{eqn:variational}) are 
equivalent to the following recursively defined variational conditions:
\begin{equation}
A_k = I_{k+1}^k A_{k+1} I_k^{k+1},
\ \ \ \ \ I_{k+1}^k = (I_k^{k+1})^T,
   \label{eqn:variational2}
\end{equation}
when the composite operators $I_k$ appearing in~(\ref{eqn:variational})
are defined as in~(\ref{eqn:composite}).

In a finite element setting, conditions~(\ref{eqn:variational2})
can be shown to hold in ideal situations, for both the stiffness matrices 
and the abstract weak form operators, for a nested sequence of 
successively refined finite element meshes.
In the finite difference or finite volume method setting, 
conditions~(\ref{eqn:variational2}) must often be imposed algebraically,
in a recursive fashion.

The error propagator of a multiplicative V-cycle MG method is defined 
implicitly:
\begin{equation}\label{errorpropMG}
E = I- BA =  I - D_J A_J ,
\end{equation}
where $A_J= A$, and where operators $D_k,~k = 2,\ldots, J$ are 
defined recursively,
\begin{eqnarray}
I - D_k A_k &=&
    (I - \bar{R}_kA_k) (I - I_{k-1}^k D_{k-1} I_k^{k-1}A_k) (I - R_k A_k),
     \ k = 2,\ldots,J, \label{errorpropMG2} \\
D_1 &=&  R_1~. \label{errorpropMG3}
\end{eqnarray}
Operators $\bar{R}_k$ and $R_k$ are linear operators on $\calg{H}_k$,
usually called {\em smoothers}.
The linear operators $A_k\in L(\calg{H}_k,\calg{H}_k)$ define the
coarse problems.
They often satisfy the variational condition (\ref{eqn:variational2}).

The error propagator (\ref{errorpropMG}) can be thought of as an 
operator of the form (\ref{product_operator}) with
$$
\bar{B}_1 = \bar{R}_J~,~~
B_0 = I_{J-1}^J D_{J-1} I_J^{J-1}~,~~
B_1 = R_J~.
$$
Such an identification with the product method allows for
use of the result in Lemma~\ref{lemma:main_mult}.
The following theorem establishes sufficient conditions for the 
subspace operators $R_k$, $\bar{R}_k$ and $A_k$ in order to generate 
an (implicitly defined)
SPD operator $B$ that can be accelerated with conjugate gradients.
\begin{theorem}
   \label{theo:main_mult_mg}
Sufficient conditions for symmetry and positivity of the multiplicative
multigrid operator $B$, implicitly defined by (\ref{errorpropMG}), 
(\ref{errorpropMG2}), and (\ref{errorpropMG3}), are
\begin{enumerate}
\item $A_k$ is SPD on $\calg{H}_k~,~~k=2,\ldots,J$~;
\item $I_k^{k-1} = c_k (I_{k-1}^k)^T, \ \ \ c_k > 0, \ \ \ k=2,\ldots,J$~;
\item $\bar{R}_k = R_k^T, \ \ \ k=2,\ldots,J$~;
\item $R_1 = R_1^T$~;
\item $\left\| I - R_J A \right\|_A < 1$,~;
\item $\left\| I - R_k A_k \right\|_{A_k} \le 1$, \ \ \ $k=2,\ldots,J-1$~;
\item $R_1$ non-negative on $\calg{H}_1$~.
\end{enumerate}
\end{theorem}

\begin{proof}
Since $\bar{R}_J = R_J^T$, we have that $\bar{B}_1 = B_1^T$, which gives 
condition~1 of Lemma~\ref{lemma:main_mult}.
Now, $B_0$ is symmetric if and only if
$$
B_0 = I_{J-1}^J D_{J-1} I_J^{J-1} 
    = (c_J^{-1} I_J^{J-1})^T D_{J-1}^T (c_J I_{J-1}^J)^T
    = B_0^T,
$$
which holds under condition 2 and a symmetry requirement for $D_{J-1}$.
We will prove that $D_{J-1} = D_{J-1}^T$ by induction.
First, $D_1 = D_1^T$ since $R_1 = R_1^T$.
By Lemma~\ref{lemma:TWO} and condition 1, $D_k$ is symmetric if and only if 
$E_k = I- D_k A_k$ is $A_k$-self-adjoint.
By using (\ref{adjointofM}), we have that
\begin{eqnarray*}
E_k^* & = & A_k^{-1} \left(
      (I - \bar{R}_k A_k) (I - I_{k-1}^k D_{k-1} I_k^{k-1} A_k) (I - R_k A_k)
      \right)^T A_k \\
      & = & A_k^{-1} 
             (I - A_k^T R_k^T ) 
             (I - A_k^T (I_k^{k-1})^T D_{k-1}^T (I_{k-1}^k)^T )
             (I - A_k^T \bar{R}_k^T )
            A_k \\
      & = & (I - R_k^T A_k) 
            A_k^{-1} (I - A_k^T (I_k^{k-1})^T D_{k-1}^T (I_{k-1}^k)^T ) A_k
            (I - \bar{R}_k^T A_k) \\
      & = & (I - \bar{R}_k A_k) 
            (I - (c_k I_{k-1}^k) D_{k-1}^T (c_k^{-1} I_k^{k-1}) A_k)
            (I - R_k A_k)~,
\end{eqnarray*}
where we have used conditions 1, 2 and 3.
Therefore, $E_k^* = E_k$,  if $D_{k-1} = D_{k-1}^T$.
Hence, the result follows by induction on $k$.

Condition~3 of Lemma~\ref{lemma:main_mult} follows trivially by condition~5 of 
the theorem.

It remains to verify condition~4 of Lemma~\ref{lemma:main_mult}, 
namely that $B_0$ is non-negative. This is equivalent to showing that
$D_{J-1}$ is non-negative on $\calg{H}_{J-1}$. 
This will follow again from an induction argument.
First, note that $D_1 = R_1$ is non-negative on $\calg{H}_1$.
Next, we prove that $(D_k v_k, v_k)\geq 0$, $\forall v_k \in \calg{H}_k$, or,
equivalently, since $A_k$ is non-singular, that $(D_k A_kv_k, A_kv_k)\geq 0$.
So, for all $v_k \in \calg{H}_k$,
\begin{eqnarray*}
(D_k A_k v_k, A_k v_k) 
   & = & (A_k v_k, v_k) - (A_k E_k v_k, v_k) \\
   & = & (A_k v_k, v_k) \\
&&   - (A_k (I-\bar{R}_k A_k) (I- I_{k-1}^k D_{k-1} I_k^{k-1} A_k) (I-R_kA_k) 
         v_k,v_k) \\
   & = & (A_k v_k, v_k) \\
&&   - (A_k (I - I_{k-1}^k D_{k-1} I_k^{k-1} A_k) (I-R_kA_k) v_k, 
     (I-R_kA_k)v_k) \\
   & = & (A_k v_k, v_k) - (A_k (I-R_kA_k) v_k, (I-R_kA_k)v_k) \\
   & & +~(A_k I_{k-1}^k D_{k-1} I_k^{k-1} A_k (I-R_kA_k) v_k, (I-R_kA_k)v_k) \\
   & = & (v_k,v_k)_{A_k} - (S_k v_k, S_k v_k)_{A_k} 
       + c_k^{-1} (D_{k-1} v_{k-1},v_{k-1})
\end{eqnarray*}
where 
$S_k = I-R_kA_k$ and $v_{k-1}=I_k^{k-1} A_k (I-R_kA_k) v_k\in\calg{H}_{k-1}$.
By condition 6, the first two terms in the above sum add up to a non-negative 
value.
Hence, $D_k$ is non-negative if $D_{k-1}$ is non-negative.
Condition 4 of Lemma ~\ref{lemma:main_mult} follows.
\end{proof}
\begin{corollary}
If the fine grid smoother is symmetric, i.e., $R_J = \bar{R}_J^T$, then 
condition~5 in Theorem~\ref{theo:main_mult_mg} is equivalent 
to $\rho(I-R_J A)<1$.
\end{corollary}
\begin{proof}
This follows directly from Corollary \ref{corr_symm}.
\end{proof}

\begin{remark} \label{remark:mg_mult_galerkin}
The coarse grid operators $A_k$, $k=2,\ldots,J-1$, need only be SPD. 
They need not satisfy the Galerkin conditions~(\ref{eqn:variational2}).
\end{remark}

\begin{remark}
As noted earlier in Remark~\ref{remark:matrices}, the conditions
and conclusions of the theorem can be interpreted completely in terms of the
usual matrix representations of the multigrid operators.
\end{remark}

\begin{remark}
Condition~1 of the theorem requires that the coarse grid operators 
(except for the coarsest one) be SPD.
This is easily satisfied when they are constructed either by 
discretization or by explicitly using the Galerkin or variational condition.
Condition~2 requires restriction and prolongation to be adjoints, 
possibly multiplied by an arbitrary constant.
Condition~3 of the theorem is satisfied when 
the number of pre-smoothing steps equals the number of post-smoothing steps, 
and in addition one of the following is imposed:
(1) use of the same symmetric smoother for both pre- and post-smoothing;
(2) use of the adjoint of the pre-smoothing operator as the post-smoother.
Condition~4 requires a symmetric coarsest mesh solver.
When the coarsest mesh problem is SPD, the symmetry of $R_1$ is 
satisfied when it corresponds to an exact solve (as is typical for MG methods).
Condition~5 is a convergence requirement on the fine space smoother.
Condition~6 requires the coarse grid smoothers to be non-divergent.
The nonnegativity requirement for $R_1$ is a non-trivial one; however, if
$A_1$ is SPD, it is immediately satisfied when the operator corresponds to 
an exact solve. 
\end{remark}

Theorem \ref{theo:main_mult_mg} applies to standard multigrid methods only.
The conditions of the theorem, and condition~5 in particular, cannot be 
satisfied in the cases of hierarchical basis multigrid methods~\cite{BDY88}, 
and multigrid methods with local smoothing on locally refined regions.  
The latter methods are covered in the following theorem, where the conditions
that guarantee positivity of the preconditioner (conditions~5,6 and 7 in
Theorem~\ref{theo:main_mult_mg}), are replaced by a convergence condition
on the underlying iterative method.
\begin{theorem}
   \label{corr:main_mult_mg_nonoverlap}
Sufficient conditions for symmetry and positivity of the multiplicative
multigrid operator $B$, implicitly defined by (\ref{errorpropMG}), 
(\ref{errorpropMG2}), and (\ref{errorpropMG3}), are
\begin{enumerate}
\item $A_k$ is SPD on $\calg{H}_k~,~~k=2,\ldots,J$~;
\item $I_k^{k-1} = c_k (I_{k-1}^k)^T, \ \ \ c_k > 0, \ \ \ k=2,\ldots,J$~;
\item $\bar{R}_k = R_k^T, \ \ \ k=2,\ldots,J$~;
\item $R_1 = R_1^T$~;
\item $\left\| I - B A \right\|_A < 1$~.
\end{enumerate}
\end{theorem}
\begin{proof}
Positivity of $B$ is proven easily by a contradiction argument.
Symmetry follows from the proof of Theorem~\ref{theo:main_mult_mg}.
\end{proof}

Requiring convergence of the underlying multigrid method is very 
restrictive; it is not a necessary condition.
Positivity of $B$ is satisfied if $\lambda_i(I-BA) <1$; 
no limit needs to be set on the magnitude of the negative eigenvalues.
The above eigenvalue condition, however, does not seem to lead to conditions 
that are easily checked in practice.

\begin{remark}
If variational conditions are satisfied on all levels, then there is a simple 
proof which shows that in addition to defining an SPD operator $B$, the 
conditions of Theorem~\ref{theo:main_mult_mg} are sufficient to prove the 
convergence of the MG method itself.
The result is as follows.
\begin{theorem}
   \label{theo:simple_multilevel}
If in addition to the conditions for Theorem~\ref{theo:main_mult_mg}, 
it holds that $A_k = I^k A I_k$, $I^k =  I_k^T$, and $R_1 = A_1^{-1}$,
then the MG error propagator satisfies:
$$
\rho(E) \le \| E \|_A < 1. 
$$
\end{theorem}
\begin{proof}
Under the conditions of the theorem, the MG error propagator can be 
written explicitly as the product (\cite{BPWX91b,McRu83}):
$$
E = (I - I_J R_J^T I_J^T A)   \cdots (I - I_1 R_1 I_1^T A)
                              \cdots (I - I_J R_J I_J^T A)~.
$$
Since the coarse problem is solved exactly, and since variational conditions 
hold, the coarse product term is an $A$-orthogonal projector:
$$
I - I_1 R_1 I_1^T A = I - I_1 (I_1 A I_1^T)^{-1} I_1^T A
= (I - I_1 (I_1 A I_1^T)^{-1} I_1^T A)^2 = (I - I_1 R_1 I_1^T A)^2.
$$
Therefore, we may define 
$\bar{E} = (I - I_1 R_1 I_1^T A)\cdots (I-I_J R_J I_J^T A)$,
and represent $E$ as the product $E = \bar{E}^* \bar{E}$.
Now, since $A$ is SPD, we have that:
$$
(A E v,v) = (A \bar{E} v, \bar{E} v) \ge 0~.
$$
Hence, $E$ is $A$-non-negative.
Under the conditions of the theorem, Lemma~\ref{lemma:main_mult} implies that 
the preconditioner is SPD, and so by Lemma~\ref{lemma:rho3} it holds
that $ \|E \|_A < 1$.
\end{proof}
\end{remark}
\section{Additive Schwarz methods}

We now present an analysis of additive Schwarz methods.
We establish sufficient conditions for additive algorithms to yield 
SPD preconditioners.
This theory is then employed to establish sufficient SPD conditions for 
additive DD and MG methods.

\subsection{A sum operator}

Consider a sum operator of the following form:
\begin{equation}\label{sum_operator}
E = I - BA = I - \omega (B_0 + B_1) A, ~~~~ \omega > 0~,
\end{equation}
where, as before, $A$ is an SPD operator, and $B_0$ and $B_1$ are 
linear operators on $\calg{H}$.
\begin{lemma}
     \label{lemma:main_add}
Sufficient conditions for symmetry and positivity of $B$, defined in
(\ref{sum_operator}), are 
\begin{enumerate}
\item $B_1$ is SPD in $\calg{H}$~;
\item $B_0$ is symmetric and non-negative on $\calg{H}$~.
\end{enumerate}
\end{lemma}
\begin{proof}
We have that $B = \omega (B_0 + B_1)$, which is symmetric by
the symmetry of $B_0$ and $B_1$.
Positivity follows since $(B_0 u,u) \ge 0$ and $(B_1u,u)>0$, 
$\forall u \in \calg{H}$, $u\ne 0$.
\end{proof}

\begin{remark} \label{remark:alpha2}
The parameter $\omega$ is usually required to make the additive method
a convergent one.  Its estimation is often nontrivial, and can be 
very costly.
As was noted in Remark~\ref{remark:alpha}, the parameter
$\omega$ is not required when the linear additive method is used
as a preconditioner in a conjugate gradients algorithm.
This is exactly why additive multigrid and domain decomposition methods
are used almost exclusively as preconditioners.
\end{remark}

\subsection{Additive domain decomposition}

As in \S\ref{sect_mdd}, we consider the Hilbert space $\calg{H}$, and
$J$ subspaces $I_k \calg{H}_k$ such that 
$I_k \calg{H}_k \subseteq \calg{H} = \sum_{k=1}^J I_k \calg{H}_k$.
Again, we allow for the existence of a ``coarse'' subspace 
$I_0\calg{H}_0 \subseteq \calg{H}$.

The error propagator of an additive DD method on the space 
$\calg{H}$ employing the subspaces $I_k \calg{H}_k$ has the general
form~(see \cite{Xu92a}):
\begin{equation}\label{add_propagator}
E = I - BA 
  = I - \omega (I_0 R_0 I^0 + I_1 R_1 I^1 + \cdots + I_J R_J I^J) A.
\end{equation}
The operators $R_k$ are linear operators on $\calg{H}_k$,
constructed in such a way that $R_k\approx A_k^{-1}$, where the $A_k$ are the
subdomain problem operators.
Propagator (\ref{add_propagator}) can be thought of as the sum method 
(\ref{sum_operator}), by taking
$$
B_0 = I_0 R_0 I^0,
\ \ \ \ \ B_1 = \sum_{k=1}^J I_k R_k I^k.
$$

This identification allows for the use of Lemma~\ref{lemma:main_add} in order
to establish conditions to guarantee that additive domain 
decomposition yields an SPD preconditioner.
Before we state the main theorem, we need the following lemma,
which characterizes the splitting of $\calg{H}$ into the 
subspaces $I_k \calg{H}_k$ in terms of a positive 
{\it splitting constant} $S_0$.
\begin{lemma}
  \label{lemma:split_dd}
Given any ~$v \in \calg{H}$, there exists 
a splitting $v=\sum_{k=1}^J I_k v_k$, $v_k \in \calg{H}_k$, and a constant 
$S_0 >0$, such that
\begin{equation}\label{S0_bound}
\sum_{k=1}^J \| I_k v_k \|_A^2 \le S_0 \| v \|_A^2.
\end{equation}
\end{lemma}
\begin{proof}
Since $ \sum_{k=1}^J I_k \calg{H}_k = \calg{H}$,
we can construct subspaces $\calg{V}_k\subseteq \calg{H}_k$, such that
$$
I_k\calg{V}_k \cap I_l\calg{V}_l = \{ 0 \}~, ~\mbox{for}~k\neq l
~~\mbox{and}~~\calg{H} = \sum_{k=1}^J I_k \calg{V}_k~.
$$
Any $v \in \calg{H}$, can be decomposed uniquely as 
$v=\sum_{k=1}^J I_k v_k$, $v_k \in \calg{V}_k$. Define the projectors
$Q_k\in L(\calg{H},I_k\calg{V}_k)$ such that $ Q_k v = I_kv_k$.
Then,
$$
\sum_{k=1}^J \| I_k v_k \|_A^2
   = \sum_{k=1}^J \| Q_k v \|_A^2
   \leq  \sum_{k=1}^J ~\| Q_k\|_A^2 ~\| v \|_A^2~.
$$  
Hence,  the result follows with $S_0 = \sum_{k=1}^J \| Q_k\|_A^2$.
\end{proof}

\begin{theorem}
   \label{theo:main_add_dd}
Sufficient conditions for symmetry and positivity of the additive 
domain decomposition operator $B$, defined in (\ref{add_propagator}), are
\begin{enumerate}
\item $I^k = c_k I_k^T$, ~~$c_k > 0$, ~~$k=0,\dots,J$~;
\item $R_k$ is SPD on $\calg{H}_k~,~~k=1,\dots,J$~;
\item $R_0$ is symmetric and non-negative on $\calg{H}_0$~.
\end{enumerate}
\end{theorem}

\begin{proof}
Symmetry of $B_0$ and $B_1$ follow trivially from the symmetry of $R_k$ and 
$R_0$, and from $I^k=c_k I_k^T$.
That $B_0$ is non-negative on $\calg{H}$ follows immediately from the 
non-negativity of $R_0$ on $\calg{H}_0$.

Finally, we prove positivity of $B_1$.
Define (only as a technical tool) the operators $A_k = I^k A I_k$, 
$k=1,\ldots,J$.
By condition~1, and the full rank nature of $I_k$, 
we have that $A_k$ is SPD.
Now, since $R_k$ is also SPD, the product $R_k A_k$ is 
$A_k$-SPD.  Hence, there exists an $\omega_0 > 0$ such that
$0 < \omega_0 < \lambda_i(R_k A_k), ~k = 1, \ldots, J$.
This is used together with  (\ref{S0_bound}) to bound the following sum,
$$
\sum_{k=1}^J c_k^{-1} (R_k^{-1} v_k,v_k) 
   = \sum_{k=1}^J c_k^{-1} (A_k A_k^{-1} R_k^{-1} v_k,v_k)
$$
$$
\le \sum_{k=1}^J c_k^{-1} (A_k v_k, v_k)
      \max_{v_k \ne 0}\frac{(A_k A_k^{-1} R_k^{-1} v_k,v_k)}{(A_k v_k, v_k)}
\le \sum_{k=1}^J c_k^{-1} \omega_0^{-1} (A_k v_k, v_k) 
$$
$$
=  \sum_{k=1}^J \omega_0^{-1} ( A I_k v_k, I_k v_k)
=  \sum_{k=1}^J \omega_0^{-1} \| I_k v_k \|_A^2 
\le \left( \frac{S_0}{\omega_0} \right) \| v \|_A^2,
$$
with $v = \sum_{k=1}^J I_k v_k$.
We can now employ this result to establish positivity of $B_1$.
First note that
\begin{displaymath}
\| v \|_A^2 = (Av,v) = \sum_{k=1}^J (Av,I_k v_k)
  = \sum_{k=1}^J (I_k^T Av,v_k)
  = \sum_{k=1}^J (R_k c_k^{1/2} I_k^T Av, R_k^{-1} c_k^{-1/2} v_k) ~.
\end{displaymath}
By using the Cauchy-Schwarz inequality first in the $R_k$-inner-product
and then in $\bbbb{R}^J$, we have that
\begin{eqnarray*}
\| v \|_A^2  &\le& \left( \sum_{k=1}^J 
     (R_k R_k^{-1} c_k^{-1/2} v_k, R_k^{-1} c_k^{-1/2} v_k) \right)^{1/2}
      \left( \sum_{k=1}^J 
      (R_k c_k^{1/2} I_k^T A v, c_k^{1/2} I_k^T A v) \right)^{1/2} \\
&\le& \left( \frac{S_0}{\omega_0} \right)^{1/2}~ \| v \|_A~
      \left( \sum_{k=1}^J (I_k R_k c_k I_k^T A v, A v) \right)^{1/2} \\
&=&   \left( \frac{S_0}{\omega_0} \right)^{1/2} 
        \| v \|_A~(B_1 A v, A v)^{1/2}~. 
\end{eqnarray*}
Finally, division by $\| v \|_A$ and squaring yields
$$
(B_1 Av, Av)~\ge~\frac{\omega_0}{S_0}~\| v\|_A^2 > 0~,~~\forall 
v\in\calg{H} ~,~~ v \ne 0~.
$$
\end{proof}

\begin{remark} \label{remark:dd_add_galerkin}
Condition~1 is naturally satisfied for $k=1,\ldots,J$,
with $c_k = 1$, since the associated $I_k$ and $I^k$ are usually 
inclusion and orthogonal projection operators (which are natural adjoints
when the inner-products are inherited from the parent space, as in
domain decomposition).
The fact that $I^0= c_0 I_0^T$ needs to be satisfied explicitly.
Condition~2 requires the use of SPD subdomain solvers.  
The condition will hold, for example, when the subdomain solve is exact
and the subdomain problem solver is SPD.
(The latter is naturally satisfied by condition~1 and the full rank 
nature of $I_k$.)
Finally, condition~3 is nontrivial, and needs to be checked explicitly.
The condition holds when the coarse space problem operator is SPD and the 
solve is exact.
Note that variational conditions are not needed for the coarse space 
problem operator.
\end{remark}

Consider again the case when the subspaces together do not 
span the entire space, except when the coarse space is included.
The above theorem applies immediately with $R_0=0$, where now the coarse
space is taken to be any one of the spaces $\calg{H}_k$, $k\ne 0$.
The error propagator takes the form
\begin{equation}\label{add_propagator_v2}
I - BA 
  = I - \omega (I_1 R_1 I^1 + I_2 R_2 I^2 + \cdots + I_J R_J I^J) A.
\end{equation}
This leads to the following corollary.
\begin{corollary}
   \label{corr:main_add_dd_nonoverlap}
Sufficient conditions for symmetry and positivity of the additive 
domain decomposition operator $B$, defined in (\ref{add_propagator_v2}), are
\begin{enumerate}
\item $I^k = c_k I_k^T$, ~~$c_k > 0$, ~~$k=1,\dots,J$~;
\item $R_k$ is SPD on $\calg{H}_k~,~~k=1,\dots,J$~.
\end{enumerate}
\end{corollary}

\subsection{Additive multigrid}

As in \S\ref{sect_mmg}, given are the Hilbert space $\calg{H}$, and 
$J-1$  nested subspaces $I_k \calg{H}_k$ such that 
$ I_1 \calg{H}_1 \subseteq I_2 \calg{H}_2 
   \subseteq \cdots \subseteq I_{J-1} \calg{H}_{J-1} 
   \subseteq \calg{H}_J \equiv \calg{H}~$.
The operators $I_k$ and $I^k$ are the usual linear operators between the 
different spaces, as in the previous sections.

The error propagator of an additive MG method is defined explicitly:
\begin{equation}\label{propagator_mmg}
E = I - BA 
  = I - \omega (I_1 R_1 I^1 + I_2 R_2 I^2 + \cdots 
    +  I_{J-1} R_{J-1} I^{J-1} + R_J ) A.
\end{equation}
This can be thought of as the sum method analyzed earlier, by taking
$$
B_0 = \sum_{k=1}^{J-1} I_k R_k I^k~,~~B_1 = R_J ~.
$$

This identification allows for the use of Lemma~\ref{lemma:main_add} to 
establish sufficient conditions to guarantee that additive MG 
yields an SPD preconditioner.
\begin{theorem}
   \label{theo:main_add_mg}
Sufficient conditions for symmetry and positivity of the additive multigrid
operator $B$, defined in (\ref{propagator_mmg}), are:
\begin{enumerate}
\item $I^k = c_k I_k^T~, ~~c_k > 0~, ~~k=1,\ldots,J-1$~;
\item $R_J$ is SPD~ in $\calg{H}$~;
\item $R_k$ is symmetric non-negative in $\calg{H}_k~,~~k=1,\ldots,J-1$~.
\end{enumerate}
\end{theorem}

\begin{proof}
Symmetry of $B_0$ and $B_1$ is obvious.
$B_1$ is positive by condition~2.
Non-negativity of $B_0$ follows from
$$ 
(B_0 u, u) = \sum_{k=1}^{J-1} (I_k R_k (c_k I_k)^T u, u)
           = \sum_{k=1}^{J-1} c_k (R_k I_k^T u, I_k^T u)
           \ge 0, \ \ \ \forall u \in \calg{H}, u \ne 0.
$$ 
\end{proof}

\begin{remark} \label{remark:mg_add_galerkin}
Condition~1 of the theorem has to be imposed explicitly.
Conditions~2 and~3 require the smoothers to be symmetric.
The positivity of $R_J$ is satisfied when the fine grid smoother is 
convergent, although this is not a necessary condition.
The non-negativity of $R_k$, $k<J$, has to be checked explicitly.
When the coarse problem operators are SPD, this condition is
satisfied, for example, when the smoothers are non-divergent.
Note that variational conditions for the subspace problem operators 
are not required.
\end{remark}

Theorem~\ref{theo:main_add_mg} is applicable to the standard 
multigrid case, i.e., the case where the fine mesh smoother
operates on the entire fine mesh. 
As in \S3.3, a different set of conditions is to be derived to 
cover the cases of (additive) hierarchical basis 
preconditioners~\cite{Yser86b}, and additive multilevel methods with smoothing 
in local refinement regions only.
The latter cases are treated most easily by loosening the restriction 
of nestedness of the spaces $I_k \calg{H}_k$.
With $\calg{H}_k$ interpreted as the domain space of the smoother $R_k$,
the theory becomes identical to that of the additive domain 
decomposition case.
Sufficient conditions for the additive operator $B$ to be SPD are
then similar to the conditions of the 
additive domain decomposition method in 
Corollary~\ref{corr:main_add_dd_nonoverlap}

\section{To symmetrize or not to symmetrize}

\label{sec:symmetry}

The following lemma illustrates why symmetrizing is a bad idea for linear 
methods. 
It exposes the convergence rate penalty incurred by symmetrization of a 
linear method.

\begin{lemma}
   \label{lemma:penalty}
For any $E \in \bold{L}(\calg{H},\calg{H})$, it holds that:
$$
\rho(EE) \le \|E E\|_A \le \|E\|_A^2 
           = \| E E^* \|_A = \rho(E E^*).
$$
\end{lemma}
\begin{proof}
The first and second inequalities hold for any norm.
The first equality follows from Lemma~\ref{lemma:Anorm_Esquare_EEstar}, 
and the second follows from Lemma~\ref{lemma:a-s-a}.
\end{proof}

Note that this is an inequality not only for the spectral radii, which is 
only an asymptotical measure of convergence, but also for the $A$-norms of the 
nonsymmetric and symmetrized error propagators. 
The lemma illustrates that one may actually see the differing convergence 
rates early in the iteration as well.

Based on this lemma, and Corollary~\ref{coro:cond2}, we conjecture that when 
symmetrization of a linear method is required for its use as a 
preconditioner, the best results will be obtained by enforcing only a minimal 
amount of symmetry.
This conjecture can be made more clear by considering the two linear methods
with error propagators $E_1 = E E$ and $E_2 = E E^*$, and the effectiveness as 
preconditioners of the symmetrized product operators 
$E_1 E_1^*$ and $E_2 E_2^*$.
Using Lemma~\ref{lemma:a-s-a} and Lemma~\ref{lemma:penalty},
we find immediately that
\begin{equation}
   \label{eqn:joe}
\| E_1 E_1^* \|_A
  = \| E_1 \|_A^2
  \le \| E_2 \|_A^2
  = \| E_2 E_2^* \|_A.
\end{equation}
While both operators are symmetric, the operator $E_1 E_1^*$ has been 
symmetrized ``minimally'' in the sense that the individual terms $E_1$ making 
up the product are themselves nonsymmetric.
On the other hand, both of the terms $E_2$ making up the operator $E_2 E_2^*$ 
are completely symmetric.  

As a result of inequality~(\ref{eqn:joe}) and Corollary~\ref{coro:cond2}, 
the bound for the condition number of the preconditioner associated with 
$E_1 E_1^*$ is less than the corresponding bound for the preconditioner
associated with $E_2 E_2^*$.
Hence, the ``less-symmetric'' $E_1 E_1^*$ would likely produce a better 
preconditioner than the ``more-symmetric'' $E_2 E_2^*$.

\section{Numerical results} \label{sec:numerical}

We present numerical results obtained by using multiplicative and additive
finite-element-based DD and MG methods applied to two test problems,
and we illustrate the theory of the preceding sections.

\subsection{Example 1}

Violation of variational conditions can occur in DD and MG methods when, for 
example, complex coefficient discontinuities do not lie along element 
boundaries on coarse meshes.
An example of this occurs with the following test problem.
The Poisson-Boltzmann equation describes the electrostatic potential of a 
biomolecule lying in an ionic solvent (see, e.g.,~\cite{BrMc90} 
for an overview).
This nonlinear elliptic equation for the dimensionless electrostatic potential 
$u(\bold{r})$ has the form:
$$
- \nabla \cdot ( \epsilon(\bold{r}) \nabla u(\bold{r}) )
  + \bar{\kappa}^2 
      \sinh ( u(\bold{r}) )
  = \left( \frac{4 \pi e_c^2}{k_B T} \right)
     \sum_{i=1}^{N_m} z_i \delta(\bold{r} - \bold{r}_i),
  \ \ \bold{r} \in \bbbb{R}^3,
\ \ \ ~u(\infty) = 0~.
$$
The coefficients appearing in the equation are discontinuous by 
several orders of magnitude.
The placement and magnitude of atomic charges are represented by source 
terms involving delta-functions.
Analytical techniques are used to obtain boundary conditions on a finite
domain boundary.

We will compare several MG and DD methods for a two-dimensional, linearized 
Poisson-Boltzmann problem, modeling a molecule with three point charges.
The surface of the molecule is such that the discontinuities do not align 
with the coarsest mesh or with the subdomain boundaries.
Beginning with the coarse mesh shown on the left in 
Figure~\ref{fig:ex1_mg_refine}, we uniformly refine the initial mesh of 10 
elements (9 nodes) 
four times, leading to a fine mesh of 2560 elements (1329 nodes). 
Piecewise linear finite elements, combined with one-point
Gaussian quadrature, are used to discretize the problem.
The three coarsest meshes used to formulate the MG methods are
 given in Figure~\ref{fig:ex1_mg_refine}.
For the DD methods, the  subdomains, corresponding to the initial coarse
triangulation, are given a small overlap of one fine mesh triangle.
The DD methods also employ a coarse space constructed from the 
initial triangulation.
Figure~\ref{fig:ex1_dd_subdomains} shows three overlapping subdomains 
overlaying the initial coarse mesh.

Computed results are presented in Tables~\ref{table:ex1_mg_mult} to 
\ref{table:ex1_dd_add}.  Given for each experiment is the number of 
iterations required to satisfy the error criterion 
(reduction of the $A$-norm of the error by $10^{-10}$).
We report results for the unaccelerated, CG-accelerated, and 
Bi-CGstab-accelerated methods.
Since the cost of one iteration differs for each method, 
Table~\ref{tab:ex1:itcost2} gives the operation counts per iteration, 
normalized by the cost of a single multigrid iteration.
For the MG operation counts, two smoothing iterations by 
lexicographic Gauss-Seidel (one pre- and 
one post-smoothing) are used.
The DD operation counts are for methods employing two sweeps through the 
subdomains, each approximate subdomain solve consisting
of four sweeps of a Gauss-Seidel iteration.

Table~\ref{tab:ex1:itcost2} shows that multiplicative MG is slightly more 
costly than additive MG, since additive MG 
requires the computation of the residual only on the finest level.
Similarly, multiplicative DD is somewhat more costly than additive DD,
due to the need to update boundary information (recompute the residual)
after the solution of each subdomain problem.
Table~\ref{tab:ex1:itcost2} should not be used to compare MG 
and DD methods for efficiency.
Similar experiments~\cite{HoSa93b} with more carefully optimized DD and MG 
methods show DD to be often competitive with MG for difficult elliptic 
equations such as those with discontinuous coefficients, although there may 
be some debate as to which approach is more effective on parallel 
computers~\cite{Vand94}.

\begin{table}[t]
\caption{Normalized operation counts per iteration, Example 1.}
\label{tab:ex1:itcost2}
\begin{center}
\small
\begin{tabular}{|c||c|c|c|} \hline
Method & UNACCEL
       & ~~~~~CG~~~~~
       & Bi-CGstab \\ \hline\hline
multiplicative MG & 1.0      & 1.4      & 2.6      \\ \hline
additive MG       & .95      & 1.3      & 2.5      \\ \hline
multiplicative DD & 3.5      & 3.8      & 7.5      \\ \hline
additive DD       & 3.1      & 3.4      & 6.7      \\ \hline
\end{tabular}
\end{center}
\end{table}

\begin{figure}[t]
\begin{center}
\mbox{\myfigpdf{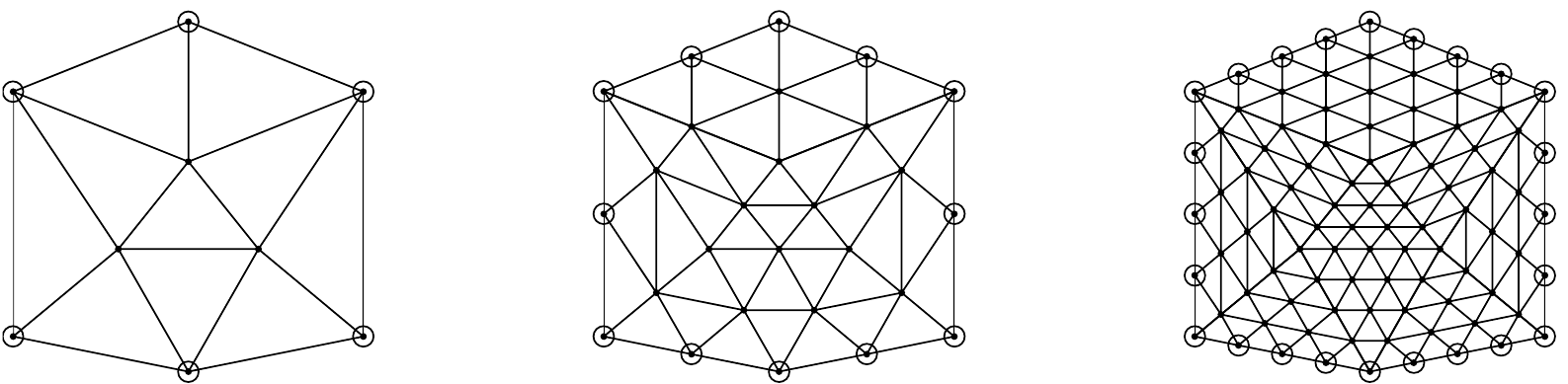}{1.0in}}
\end{center}
\caption{Example 1: Nested finite element meshes for MG.}
   \label{fig:ex1_mg_refine}

\bigskip
\begin{center}
\mbox{\myfigpdf{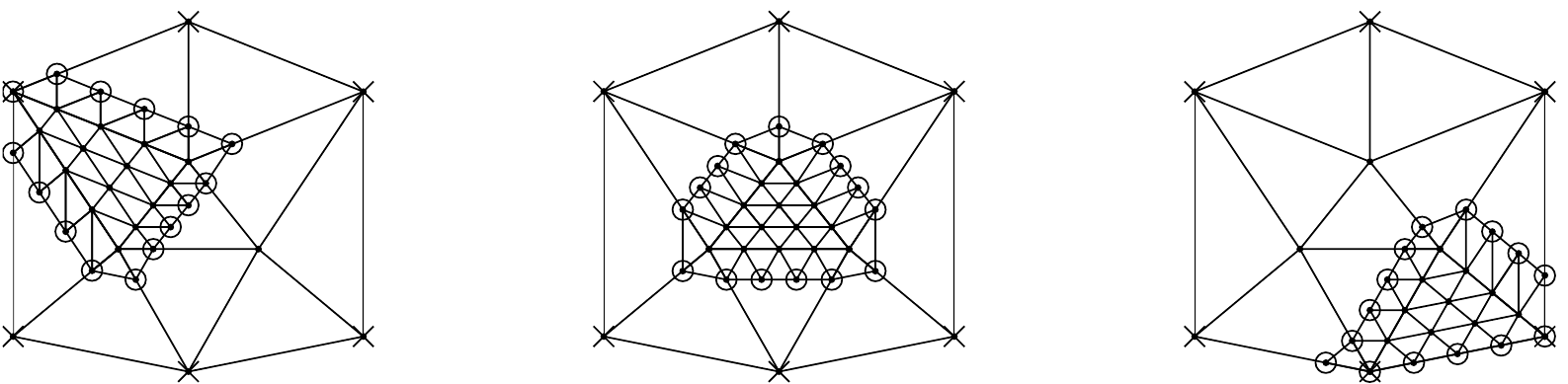}{1.0in}}
\end{center}
\caption{Example 1: Overlapping subdomains for DD.}
   \label{fig:ex1_dd_subdomains}
\end{figure}

\subsubsection*{Multiplicative multigrid}

The results for multiplicative V-cycle MG are presented in 
Table~\ref{table:ex1_mg_mult}.
Each row corresponds to a different smoothing strategy, and is annotated by
$(\nu_1,\nu_2)$, with $\nu_1$: pre-smoothing strategy, and 
$\nu_2$: post-smoothing strategy.
An ``f'' indicates the use of a single forward Gauss-Seidel sweep, while
a ``b'' denotes the use of the adjoint of the latter, i.e., a backward
Gauss-Seidel sweep.  $(\nu_1,\nu_2)$ = $(ff,fb)$, for example, 
corresponds to two forward Gauss-Seidel pre-smoothing steps,
and a symmetric (forward/backward) post-smoothing step.
Two series of results are given.  
For the first set, we explicitly imposed the Galerkin conditions when 
constructing the coarse operators.  
In this case, the multigrid algorithm is guaranteed to 
converge by Theorem~\ref{theo:simple_multilevel}.
In the second series of tests, corresponding to the numbers 
in parentheses, the coarse mesh operators are constructed using 
standard finite element discretization.   
In that case, Galerkin conditions are not satisfied everywhere due to
coefficient discontinuities appearing within coarse elements;
hence, the MG method may diverge (DIV). 

The unaccelerated MG results clearly illustrate the symmetry penalty 
discussed in \S\ref{sec:symmetry}.
The nonsymmetric methods are always superior to the symmetric ones
(the cases (f,b), (ff,bb), and (fb,fb)).  
Note that minimal symmetry (ff,bb) leads to a better convergence 
than maximal symmetry (fb,fb).
The correctness of Lemma~\ref{lemma:penalty} is illustrated by noting that two
iterations of the (f,0) strategy are actually faster than one
iteration of the (f,b) strategy; also, compare the (ff,0) strategy
to the (ff,bb) one.   
CG-acceleration leads to a guaranteed reduction in iteration count
for the symmetric preconditioners (see Lemma~\ref{theo:accel}).
We observe that the unaccelerated method need not be convergent 
for CG to be effective 
(recall Remarks~\ref{remark:alpha} and~\ref{remark:alpha2}, 
and the (f,b) result). 
CG appears to accelerate also some non-symmetric linear methods.
Yet, it seems difficult to predict failure or success beforehand in such cases.
The most robust method appears to be the  Bi-CGstab method.  
The number of iterations with this method depends only marginally on 
the symmetric or nonsymmetric nature of the linear method. 
Note the tendency to favor the nonsymmetric V-cycle strategies.
Overall, the fastest method proves to be the Bi-CGstab-acceleration of 
a (very nonsymmetric) V(1,0)-cycle.

\subsubsection*{Multiplicative domain decomposition}

Some numerical results for multiplicative DD with different subdomain solvers, 
and different subdomain sweeps are given in Table~\ref{table:ex1_dd_mult}.
In the column ``forw", the iteration counts reported were obtained with a 
single sweep though the subdomains on each multiplicative DD iteration.
The other columns correspond
to a symmetric forward/backward sweep or to two forward sweeps.  
Four different subdomain solvers are used:
an {\em exact} solve, a {\em symmetric} method consisting of two symmetric 
Gauss-Seidel iterations, a {\em nonsymmetric} method consisting of 
four Gauss-Seidel iterations, and, finally, a method using four
forward Gauss-Seidel iterations in the forward subdomain sweep and
using their {\em adjoint}, i.e., four backward Gauss-Seidel iterations,
in the backward subdomain sweep.  
The latter leads to an symmetric iteration; 
see Remark~\ref{remark:dd_mult_galerkin}.
Note that the cost of the three inexact subdomain solvers is identical.

Although apparently not as sensitive to operator symmetries as MG,
the same conclusions can be drawn for DD as for MG.
In particular, the symmetry penalty is seen for the pure DD results.
Lemma~\ref{lemma:penalty} is confirmed since two iterations in the 
column ``forw" are always more efficient than one iteration of the 
corresponding symmetrized method in column ``forw/back".
The CG results indicate that using minimal symmetry (the ``adjointed"
column) is a more effective approach than the fully symmetric
one (the ``symmetric" column).  
Again, the most robust acceleration is the Bi-CGstab one.

\subsubsection*{Additive multigrid}

Results obtained with an additive multigrid method are reported in 
Table~\ref{table:ex1_mg_add}.  The number and nature of the 
smoothing strategy is given in the first column of the table.

In the case of an unaccelerated additive method, the selection of 
a good damping parameter is crucial for convergence of the method.
We did not search extensively for an optimal parameter; a selection of 
$\omega=0.45$ seemed to provide good results in the case when the coarse
problem is variationally defined.  No $\omega$-value leading to
satisfactory convergence was found in the case when the coarse 
problem is obtained by discretization.
In the case of CG acceleration the observed convergence behavior was 
completely independent of the choice of $\omega$; 
see Remark~\ref{remark:dd_mult_galerkin}.
The symmetric methods ($\nu =  fb, ffbb, fbfb$) are accelerated very
well. Some of the nonsymmetric methods are accelerated too, especially
when the  number of smoothing steps is sufficiently large.
In the case of Bi-CGstab-acceleration, there appeared to 
be a dependence of convergence on $\omega$ (only with use of 
non-variational coarse problem).  In that case we took $\omega=1$.  
The overall best method appears to be the Bi-CGstab acceleration 
of the nonsymmetric multigrid method with a single forward Gauss-Seidel 
sweep on each grid-level.

\subsubsection*{Additive domain decomposition}

The results for additive DD are given in Table~\ref{table:ex1_dd_add}.
The subdomain solver is either an exact solver, a symmetric solver
based on two symmetric (forward/backward) Gauss-Seidel sweeps,
or a nonsymmetric solver based on four forward Gauss-Seidel iterations.

No value of $\omega$ was found that led to satisfactory 
convergence of the unaccelerated method. CG-acceleration performs
well when the linear method is symmetric; it performs less well
for the nonsymmetric method.
Again, the best overall method is the Bi-CGstab-acceleration of
the nonsymmetric additive solver.

\begin{table}[p]
\caption{Example 1: Multiplicative MG 
         with variational (discretized) coarse problem}
\begin{center}
\small
\begin{tabular}{|rl||rl|rl|rl|} \hline
$\nu_1$ & $\nu_2$ 
   & \multicolumn{2}{c|}{UNACCEL}  
   & \multicolumn{2}{c|}{CG}
   & \multicolumn{2}{c|}{Bi-CGstab} \\ \hline\hline
f       &    0    & 65&(DIV) & $\gg$100&($\gg$100) & 14& (16) \\ \hline \hline
f       &     b   & 55&(DIV) & 16      &(18)       & 10&(15) \\ \hline
f       &     f   & 40&(31)  & 30      &($\gg$100) &  9& (9) \\ \hline
ff      &     0   & 39&(48)  & $\gg$100&($\gg$100) &  8&(10) \\ \hline
fb      &     0   & 53&(DIV) & $\gg$100&($\gg$100) & 10&(11) \\ \hline
0       &    ff   & 39&(29)  & 29      &($\gg$100) &  8& (9) \\ \hline
0       &    fb   & 53&(DIV) & 17      &(99)       & 10&(12) \\ \hline \hline
fb      &    fb   & 34&(27)  & 12      &(13)       &  8& (8) \\ \hline
ff      &    bb   & 28&(18)  & 11      &(11)       &  7& (7) \\ \hline
ff      &    ff   & 24&(15)  & 12      &(12)       &  6& (6) \\ \hline
fff     &    f    & 24&(15)  & 17      &(27)       &  6& (6) \\ \hline
ffff    &    0    & 25&(17)  & $\gg$100&($\gg$100) &  7& (6) \\ \hline 
\end{tabular}
\end{center}
   \label{table:ex1_mg_mult}

\vspace{1cm}
\caption{Example 1: Multiplicative DD 
         with variational (discretized) coarse problem}
\begin{center}
\small
\begin{tabular}{|c|c||rl|rl|rl|} \hline
 Accel.   & subdomain solve & \multicolumn{2}{c|}{forw} & 
\multicolumn{2}{c|}{forw/back} & \multicolumn{2}{c|}{forw/forw} \\\hline\hline
 UNACCEL  & exact         & 40 & (42)   & 38 &(39)  & 20 &(21)  \\ \cline{2-8}
          & symmetric     & 279& (282)  & 146&(149) & 140&(141) \\ \cline{2-8}
          & adjointed     & -- &        & 110&(112) & 102&(103) \\ \cline{2-8}
          & nonsymmetric  & 189&  (191) & 102&(104) & 95 &(96)  \\ \hline
 CG       & exact         & $\gg$500& ($\gg$500)& 13&(13)&20&(20)\\ \cline{2-8}
          & symmetric     & 140 &  (56)  & 24&(24)&29&(27)\\ \cline{2-8}
          & adjointed     & --  &  --    & 21&(21)&25&(26)\\ \cline{2-8}
          & nonsymmetric  & 135 & (83)   & 22&(23)&28&(28)\\ \hline
 Bi-CGstab & exact        & 9  &(9)  & 9 &(9)    & 6 &(6)  \\ \cline{2-8}
          & symmetric     & 23 &(23) & 17&(16)   & 16&(16) \\ \cline{2-8}
          & adjointed     & -- & --  & 14&(14)   & 14&(13) \\ \cline{2-8}
          & nonsymmetric  & 19 &(20) & 13&(13)   & 13&(13) \\ \hline
\end{tabular}
\end{center}
   \label{table:ex1_dd_mult}

\vspace{1cm}
\caption{Example 1: Additive MG 
         with variational (discretized) coarse problem}
\begin{center}
\small
\begin{tabular}{|c||rl|rl|rl|} \hline
$\nu$  
   & \multicolumn{2}{c|}{UNACCEL}  
   & \multicolumn{2}{c|}{CG}
   & \multicolumn{2}{c|}{Bi-CGstab} \\ \hline\hline
f      & 175&($\gg$1000) & $\gg$100&($\gg$100) & 23&(52) \\ \hline\hline
ff     & 110&($\gg$1000) & 119     &(168)      & 19&(43) \\ \hline
fb     & 146&($\gg$1000) & 34      &(54)       & 23&(49) \\ \hline\hline
ffff   & 95 &($\gg$1000) & 28      &(67)       & 17&(37) \\ \hline 
ffbb   & 100&($\gg$1000) & 27      &(47)       & 17&(34) \\ \hline
fbfb   & 95 &($\gg$1000) & 28      &(48)       & 20&(43) \\ \hline
\end{tabular}
\end{center}
   \label{table:ex1_mg_add}
\end{table}
\begin{table}[t]
\caption{Example 1: Additive DD 
         with variational (discretized) coarse problem}
\begin{center}
\small
\begin{tabular}{|c||rl|rl|rl|} \hline 
subdomain solve 
   & \multicolumn{2}{c|}{UNACCEL}  
   & \multicolumn{2}{c|}{CG}
   & \multicolumn{2}{c|}{Bi-CGstab} \\ \hline\hline
exact        & $\gg$1000&($\gg$1000) & 34&(34) & 25&(27)       \\ \hline
symmetric    & $\gg$1000&($\gg$1000) & 57&(57) & 50&(49)       \\ \hline
nonsymmetric & $\gg$1000&($\gg$1000) & 69&(65) & 38&(41)       \\ \hline
\end{tabular}
\end{center}
   \label{table:ex1_dd_add}
\end{table}

\subsection{Example 2}

The second test problem is the Laplace equation on a semi-adapted
L-shaped domain, with Dirichlet boundary conditions chosen in such a way
that the equation has the following solution (in polar coordinates):
$$
u(r,\theta) =  \sqrt{r}~\sin(\theta/2)~,
$$
where the re-entrant corner in the domain is located at the origin.
Note that the one-point Gaussian quadrature rule which we employ to 
construct the stiffness matrix entries is an exact integrator here.
Hence, the variational conditions~(\ref{eqn:variational}) hold 
automatically between the fine space and all subdomain and coarse spaces for 
both the MG and the DD methods.

Figure~\ref{fig:ex2_mg_refine} shows a nested sequence of uniform mesh 
refinements used to formulate the MG methods. 
(A total of 5 mesh levels is used in the computation.)
Figure~\ref{fig:ex2_dd_subdomains} shows several overlapping subdomains 
constructed from a piece of the fine mesh of 9216 elements (4705 nodes)
overlaying the initial coarse mesh of 36 elements (25 nodes).

\begin{figure}[t]
\begin{center}
\begin{center}
\mbox{\myfigpdf{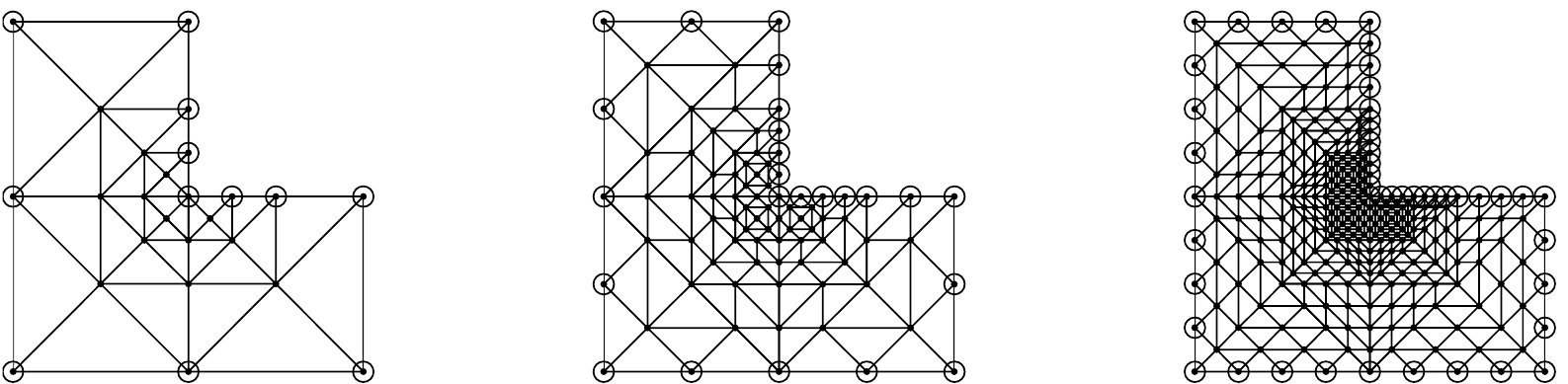}{1.0in}}
\end{center}
\caption{Example 2: Nested finite element meshes for MG.}
   \label{fig:ex2_mg_refine}

\vspace{1cm}
\mbox{\myfigpdf{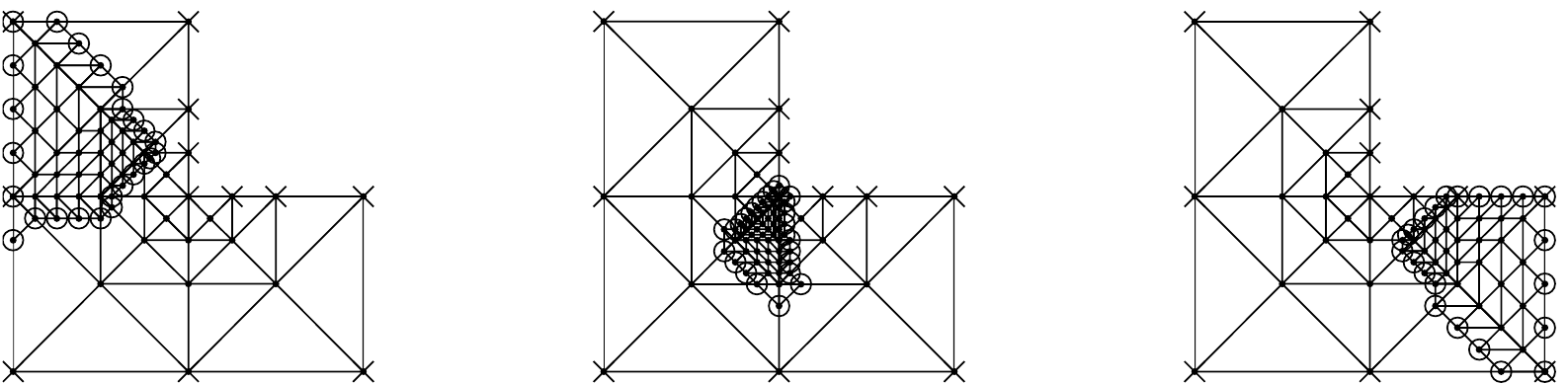}{1.0in}}
\end{center}
\caption{Example 2: Overlapping subdomains for DD.}
   \label{fig:ex2_dd_subdomains}
\end{figure}

\subsubsection*{Multiplicative Methods}

The results for multiplicative MG are given in 
Table~\ref{table:ex2_mg_mult}, whereas the results for multiplicative DD
are given in Table~\ref{table:ex2_dd_mult}.
The results are similar to those for Example 1; in particular, imposing
minimal symmetry is the most effective CG-accelerated approach to the problem.
Employing the least symmetric linear method alone is the most effective
linear method, and the same nonsymmetric linear method yields the most 
effective Bi-CGstab-accelerated approach.

\subsubsection*{Additive Methods}

As for Example~1, in the case of the unaccelerated additive methods the 
selection of the damping parameter was crucial for convergence of the methods.
We did not search extensively for an optimal parameter; a selection of 
$\omega=0.45$ seemed to provide acceptable results for DD.
Note that improved convergence behavior might be obtained by allowing
different $\omega$ values for each subdomain solver 
(this will not be further investigated here).
No satisfactory value for $\omega$ was found for additive MG.
In the case of CG acceleration, the observed convergence behavior was 
completely independent of the choice of $\omega$.
The results for additive MG are given in Table~\ref{table:ex2_mg_add},
whereas the results for additive DD are given in Table~\ref{table:ex2_dd_add}.
The effect of the symmetry of the linear method's error propagator
on its convergence, and on the convergence behavior of CG and Bi-CGstab, was
as for Example 1.

\begin{table}[p]
\caption{Example 2: Multiplicative MG}
\begin{center}
\small
\begin{tabular}{|cc||c|c|c|} \hline
$\nu_1$ & $\nu_2$ 
       & UNACCEL
       & CG
       & Bi-CGstab \\ \hline\hline
f       &    0    & 33      & $\gg$100 & 11       \\ \hline \hline
f       &     b   & 23      & 12       & 7        \\ \hline
f       &     f   & 19      & 22       & 6        \\ \hline
ff      &     0   & 21      & $\gg$100 & 7        \\ \hline
fb      &     0   & 25      & $\gg$100 & 8        \\ \hline
0       &    ff   & 20      & 42       & 7        \\ \hline
0       &    fb   & 23      & 17       & 8        \\ \hline \hline
fb      &    fb   & 16      & 9        & 6        \\ \hline
ff      &    bb   & 15      & 9        & 5        \\ \hline
ff      &    ff   & 14      & 9        & 5        \\ \hline
fff     &    f    & 14      & 12       & 5        \\ \hline
ffff    &    0    & 16      & 36       & 5        \\ \hline 
\end{tabular}
\end{center}
   \label{table:ex2_mg_mult}

\vspace{1cm}
\caption{Example 2: Multiplicative DD}
\begin{center}
\small
\begin{tabular}{|c|c||c|c|c|} \hline
Accel.    & subdomain solve & forw      & forw/back & forw/forw \\ \hline\hline
 UNACCEL  & exact          & 73   & 60  & 37  \\ \cline{2-5}
          & symmetric      & 402  & 205 & 207 \\ \cline{2-5}
          & adjointed      & --   & 153 & 146 \\ \cline{2-5}
          & nonsymmetric   & 267  & 144 & 134 \\ \hline
 CG       & exact          & 116  & 17  & 17   \\ \cline{2-5}
          & symmetric      & 164  & 37  & 38   \\ \cline{2-5}
          & adjointed      & --   & 32  & 33   \\ \cline{2-5}
          & nonsymmetric   & 121  & 31  & 32   \\ \hline
 Bi-CGstab & exact          & 11   & 11  & 7     \\ \cline{2-5}
          & symmetric      & 37   & 25  & 26    \\ \cline{2-5}
          & adjointed      & --   & 22  & 23    \\ \cline{2-5}
          & nonsymmetric   & 27   & 21  & 21    \\ \hline
\end{tabular}
\end{center}
   \label{table:ex2_dd_mult}

\vspace{1cm}
\caption{Example 2: Additive MG}
\begin{center}
\small
\begin{tabular}{|c||c|c|c|} \hline
$\nu$  
       & UNACCEL
       & CG
       & Bi-CGstab \\ \hline\hline
f      & 91        & $\gg$1000 & 21       \\ \hline\hline
ff     & 62        & 31        & 16       \\ \hline
fb     & 74        & 29        & 18       \\ \hline\hline
ffff   & 126       & 25        & 14       \\ \hline 
ffbb   & 136       & 27        & 15       \\ \hline
fbfb   & 98        & 27        & 15       \\ \hline
\end{tabular}
\end{center}
   \label{table:ex2_mg_add}
\end{table}

\begin{table}[t]
\caption{Example 2: Additive DD}
\begin{center}
\small
\begin{tabular}{|c||c|c|c|} \hline
subdomain solve 
       & UNACCEL
       & CG
       & Bi-CGstab \\ \hline\hline
exact        & $\gg$1000 & 42 & 29       \\ \hline
symmetric    & $\gg$1000 & 86 & 56       \\ \hline
nonsymmetric & $\gg$1000 & 82 & 49       \\ \hline
\end{tabular}
\end{center}
   \label{table:ex2_dd_add}
\end{table}

\section{Concluding remarks}

In this paper, we developed framework for establishing sufficient conditions 
which guarantee that abstract multiplicative and additive 
Schwarz algorithms to yield self-adjoint positive definite preconditioners.
We then analyzed four specific methods: MG and DD methods, in both their 
additive and multiplicative forms. 
In all four cases, we used the general theory to establish sufficient 
conditions that guarantee the resulting preconditioner is SPD.
As discussed in 
Remarks~
\ref{remark:dd_mult_galerkin}, 
\ref{remark:mg_mult_galerkin}, 
\ref{remark:dd_add_galerkin}, 
and~\ref{remark:mg_add_galerkin}, 
the sufficient conditions for 
the theory, in the case of all four methods, are easily satisfied for 
non-variational, and even non-convergent methods.
The analysis shows that by simply taking some care in the way a Schwarz method 
is formulated, one can guarantee that the method is convergent when 
accelerated with the conjugate gradient method.

We also investigated the role of symmetry in linear methods and 
preconditioners.
A certain penalty lemma (Lemma~\ref{lemma:penalty}) was stated and proved, 
illustrating why symmetrizing is actually a bad idea for linear methods. 
It was conjectured that enforcing minimal symmetry in a linear preconditioner 
achieves the best results when combined with the conjugate gradient method,
and our numerical examples illustrate this behavior almost uniformly.
A sequence of experiments with two non-trivial test problems showed that 
the most efficient approach may be to abandon symmetry in the
preconditioner altogether, and to employ a nonsymmetric solver such 
as Bi-CGstab.
While acceleration with CG was strongly dependent on the symmetric nature 
of the preconditioner, Bi-CGstab always converged rapidly. 
In addition, BiCGstab appeared to benefit from the behavior predicted by
Lemma~\ref{lemma:penalty}, namely that a nonsymmetric linear preconditioner 
should have better convergence properties than its symmetrized form.

\section{Acknowledgments}

The authors thank the referees and Olof Widlund for several helpful comments.

\bibliographystyle{abbrv}
\bibliography{../bib/books,../bib/papers,../bib/mjh,../bib/library,../bib/ref-gn,../bib/coupling,../bib/pnp}


\vspace*{0.5cm}

\end{document}